\documentclass[10.5pt]{article}
\usepackage{fancyhdr}

\usepackage{amsmath}
\usepackage{mathtools}
\usepackage{amsfonts}
\usepackage{amsthm}
\usepackage{amssymb}
\usepackage{color}
\usepackage{dsfont}
\usepackage{geometry} 
 \newgeometry{tmargin=2.5cm, bmargin=3.6cm, lmargin=2.2cm, rmargin=2.2cm} 

\usepackage{color}

\title{Regularizing effect of the lower-order terms\\ in elliptic  problems with Orlicz growth}

\usepackage{authblk}
 
\author[$\diamond$]{Iwona Chlebicka\thanks{The research is supported by NCN grant no. 2016/23/D/ST1/01072.}} 
\affil[$\diamond$]{\small e-mail address:  \texttt{i.chlebicka@mimuw.edu.pl}\newline Institute of Applied Mathematics and Mechanics, University of Warsaw}

\date{}
\begin{document}
\maketitle \sloppy

\thispagestyle{empty}

\renewcommand{\it}{\sl}
\renewcommand{\em}{\sl}

\belowdisplayskip=18pt plus 6pt minus 12pt \abovedisplayskip=18pt
plus 6pt minus 12pt
\parskip 4pt plus 1pt
\parindent 0pt

\newcommand{\barint}{
         \rule[.036in]{.12in}{.009in}\kern-.16in
          \displaystyle\int  } 
\def\R{{\mathbb{R}}}
\def\r{{\mathbb{R}}}
\def\n{{\mathbb{N}}}
\def\cmf{{c_{m,\mu}}}
\def\bu{{\bar{u}}}
\def\bg{{\bar{g}}}
\def\bG{{\bar{G}}}
\def\ba{{\bar{a}}}
\def\bv{{\bar{v}}}
\def\bm{{\bar{m}}}
\def\baf{{\bar{f}}}
\def\bamu{{\bar{\mu}}}
\def\bmu{{\bar{\mu}}}
\def\rn{{\mathbb{R}^{n}}}
\def\rp{{[0,\infty)}}
\def\rN{{\mathbb{R}^{N}}}
\def\avenorm#1{\mathchoice%
          {\mathop{\kern 0.2em\vrule width 0.6em height 0.69678ex depth -0.58065ex
                  \kern -0.545em \|{#1}\|}}%
          {\mathop{\kern 0.1em\vrule width 0.5em height 0.69678ex depth -0.60387ex
                  \kern -0.495em \|{#1}\|}}%
          {\mathop{\kern 0.1em\vrule width 0.5em height 0.69678ex depth -0.60387ex
                  \kern -0.495em \|{#1}\|}}%
          {\mathop{\kern 0.1em\vrule width 0.5em height 0.69678ex depth -0.60387ex
                  \kern -0.495em \|{#1}\|}}}

\newtheorem{theo}{\bf Theorem} 
\newtheorem{coro}{\bf Corollary}[section]
\newtheorem{lem}{\bf Lemma}[section]
\newtheorem{rem}{\bf Remark}[section]
\newtheorem{defi}{\bf Definition}[section]
\newtheorem{ex}{\bf Example}[section]
\newtheorem{fact}{\bf Fact}[section]
\newtheorem{prop}{\bf Proposition}[section]

\newcommand{\dv}{{\rm div}}
\newcommand{\wt}{\widetilde}
\newcommand{\ve}{\varepsilon}
\newcommand{\vp}{\varphi}
\newcommand{\vt}{\vartheta}
\newcommand{\gb}{{g^\bullet}}
\newcommand{\gbn}{{(g^\bullet_m)_n}}
\newcommand{\gbm}{{g^\bullet_m}}
\newcommand{\vr}{\varrho}
\newcommand{\pa}{\partial}
\newcommand{\MDu}{{M^*_{2B_0}(|Du|)}}
\newcommand{\Mf}{{M^*_{2B_0}(|\mu|)}}
\newcommand{\Mfg}{{M^*_{2B_0}(|f|^\gamma)}}
\newcommand{\Mmu}{{M^*_{1;2B_0}(\mu)}}
\newcommand{\Mb}{{M^*_{2B}}}
\newcommand{\Mbm}{{M^*_{1;2B}}}
\newcommand{\sgn}{{\rm sgn}}


\parindent 1em

\begin{abstract}
Under various conditions on the data we analyse how appearence of lower order terms affects the gradient estimates on solutions to a general nonlinear elliptic equation of the form
\[-\dv\, a(x,Du)+b(x,u)=\mu\]
with data $\mu$ not belonging to the dual of the natural energy space but to Lorentz/Morrey-type spaces. 
The growth of the leading part of the operator is governed by a function of Orlicz-type, whereas the lower-order term satisfies the sign condition and is minorized with some convex function, whose speed of growth modulates the regularization of the solutions.

\end{abstract}

\bigskip
\bigskip


\bigskip
\bigskip


\bigskip

\section{Introduction}
 
When data are too irregular for solutions to exist in the weak sense, we cannot expect that their very weak generalizations inherit any good regularity property. Nonetheless, infering regularity estimates for special classes of very weak solutions and their gradients in Marcinkiewicz-type spaces has already become classical~\cite{bbggpv,BocMarcSola} and has been investigated futher in Lorentz and Morrey scale too~\cite{min-grad-est}. For the corresponding  recent results  in the Orlicz setting we refer to~\cite{IC-gradest,CiMa,ACCZG}. On the other hand, it is known that the presence of the lower-order term satisfying the sign condition brings a regularizing effect on solutions to standard--growth problems~\cite{BBM,stamp1}. See also~\cite{Cirmi,DiC-Pa,Leo} and \cite{ArBo,ArBo-q} investigating the structure of the lower-term to problems with linear growth. Our aim is to extend~\cite{DiC-Pa} towards the nonstandard growth conditions providing the regularizing effect to~\cite{IC-gradest}. In fact, we investigate regularity of the solutions obtained as a limit of approximation, SOLA for short, to the problems 
\begin{equation}
\label{intro:eq:main} -\dv\, a(x,Du)+b(x,u)=\mu\quad\text{in}\quad\Omega,\end{equation} with the leading part of the operator in the Orlicz class involving the lower-order term which satisfies the sign condition and is minorized with some convex function. Let us stress from the beginning that this approach is rearrangement--free.  We provide estimates on the solutions in the scales of rearrangement invariant Lorentz-type, as well as not rearrangement invariant Morrey-type spaces, depending on the type of data. The problem we study is described in detail in Assumption {\it (Am)} in Section~\ref{sec:results}. Let us present here only the highlights and special instances. All the function spaces where the data are admitted to belong to (i.e. Lorentz $L(q,r)$, Marcinkiewicz ${\cal M}^q=L(q,\infty)$, Morrey $L^{q,\theta}$, and Lorentz-Morrey $L^\theta(q,r)$, Marcinkiewicz-Morrey ${\cal M}^{q,\theta}=L^\theta(q,\infty)$, Orlicz $L\log L$, Orlicz-Morrey $L\log L^\theta$) are defined in Appendix~\ref{ssec:fn-sp}.

 The leading part of the operator we study naturally extends the $p$-Laplace case, cf.~\cite{Lieb91,Lieb93,DieEtt}, with the additional feature that the dependence on $x$ of $a(x,\cdot)$ is just measurable. That is, in particular we admit
\[-{ \dv}\big( \omega_1(x)|\nabla u|^{p-2}\nabla u\big)+\omega_2(x)|u|^{r-1}u=\mu\quad \text{in}\quad \Omega\]
and its natural generalization
\begin{equation}
\label{intro:or-or} -{ \dv}\left(\omega_1(x)\tfrac{g(|D u|)}{|D u|}  D u \right)+\omega_2(x)\tfrac{m(u)}{|u|}u=\mu\quad\text{in}\quad \Omega,
\end{equation}
 where $\Omega$ is a bounded domain in $\rn$, $n\geq 2$, $\omega_1,\omega_2:\Omega\to [c,\infty)$ are bounded measurable and separated from zero functions, $\mu$ is  a~Borel measure with finite total mass, $|\mu|(\Omega)<\infty$, while $g,m\in C^1(0,\infty)$ are nonnegative, increasing, convex, and doubling functions, cf.~\eqref{doubling}. Note that in the case of $p$-Laplace equation we deal with  $G(t)=G_p(t)=|t|^p$ and $i_G=s_G=p\geq 2$, retrieving certain already classical results mentioned below. Indexes $i_G,s_G$ come from~\eqref{doubling}, describe the speed of growth of $G$, and play an important role in our investigations. Other examples of admissible modular functions are e.g. $G(t)=G_{p,\alpha}(t)=|t|^p\log^\alpha({\rm e}+|t|)$, their multiplications and compositions.

 Regularity of solutions to elliptic differential equations of the form 
\begin{equation}
\label{intro:eq:plap}-\Delta_p u+b(x,u)= f\quad \text{or} \quad \mu
\end{equation}
is already a well-established object of investigations, see e.g.~\cite{bgSOLA-jfa,BocMarcSola,DiC-Pa}. Let us first observe that our objective is to admit data (in general) too poorly integrable to infer existence of weak solutions. In fact, \[\text{if}\qquad \begin{cases}p>n \ \text{ or}\\
p\leq n \ \text{ and }\ \gamma>\frac{np}{np-n+p}=(p^*)',
\end{cases}\qquad \text{ then }\qquad f\in L^\gamma\subset W^{-1,p'}=(W^{1,p})^*\] and~\eqref{intro:eq:plap}  can be uniquely solved in the natural energy space, which is covered by the classical regularity theory, e.g.~\cite{DiBMan,Iw}. The most subtle case, when $p=n$, requires very delicate tools~\cite{BDSZ,DHM-p}
.  We restrict ourselves to  slowly growing operators, related to $p< n$, and small~$\gamma$. Since in such a case the notion of weak solution is  too restrictive for the chosen class of data, we employ a weaker notion of solutions, namely SOLA, defined and commented in Section~\ref{ssec:SOLA}.  There are known  estimates on gradients of solutions to~\eqref{intro:eq:plap} without the lower-order term ($b\equiv 0$) in the scales of  Marcinkiewicz/Lorentz-type and Morrey-type spaces, depending on the type of irregular data~\cite{bbggpv,BocMarcSola,min-grad-est}. See also~\cite{AdPh,Baroni-OM,CiMa14,min07}. 
As much as the issue of gradient estimates for $L^1$ or measure data is deeply investigated in the Sobolev setting, very little is known in the Orlicz spaces, where we want to contribute. To our best knowledge the related results are restricted to~\cite{ACCZG,IC-gradest,CGZG,CiMa}.   

\medskip

\textbf{On the results. } Let us start with comparison to $p$-Laplace case. We provide an analysis of the case related to $2\leq p<n$. As expected, since $G$ relates to $|t|^p,$ the role of~$p$ in the range bounds is played in the corresponding results by $i_G$ or $s_G$, cf.~\eqref{doubling}, whereas the role of $t\mapsto |t|^{p-1}$ is taken by its derivative $t\mapsto g(t)$.  The main theorem of~\cite{min-grad-est} yields the estimate in the Lorentz-Morrey setting (cf. Definition~\ref{def:LorMor:sp}) for~\eqref{intro:eq:plap} with $b\equiv 0$ and reads as follows
\begin{equation}
\label{f-Du-mingione}
\mu\in L^{\theta}(\gamma,q) \implies  |Du|^{(p-1)\frac{\theta}{\theta-\gamma}}\in L^\theta\left({\gamma},q\right)\quad\text{locally in }\ \Omega
\end{equation}
within $2\leq p<\theta\leq n$, $1<\gamma\leq {\theta p}/({\theta p - \theta +p})$, $q\in(0,\infty]$, which has been recently extended to the Orlicz version in~\cite{IC-gradest} for~\eqref{intro:eq:main} with $\omega_2\equiv 0$ to
\begin{equation}
\label{f-Du-iwonka}\mu\in L^{\theta}(\gamma,q) \implies  g^\frac{\theta}{\theta-\gamma}(|Du|)\in L^\theta\left({\gamma},q\right)\quad\text{locally in }\ \Omega
\end{equation}
within $2\leq i_G<\theta\leq n$, $1<\gamma\leq {\theta i_G}/({\theta s_G - \theta +i_G})$, $q\in(0,\infty]$.
 In~\cite{DiC-Pa} the authors adapt the method introduced in~\cite{min-grad-est} to involve the lower-order of a type $b(x,u)\geq c_0 |u|^{m_0}$ and get 
\begin{equation}
\label{f-Du-giampiero}
\mu\in L^{\theta}(\gamma,q) \implies  |Du|^{ \frac{m_0p}{m_0+1}}\in L^\theta\left(\gamma,q\right)\quad\text{locally in }\ \Omega
\end{equation}
within $2-1/n\leq p<\theta\leq n$, $1<\gamma\leq {\theta p}/({\theta p - \theta +p})$, $q\in(0,\infty]$, $p-1<m_0<1/(\gamma-1)$. We observe here the regularizing effect of the lower-order term since under this choice of parameters regularity of~\eqref{f-Du-giampiero} coming from \cite[Theorem~1]{DiC-Pa} is better than obtained in \eqref{f-Du-mingione} coming from \cite[Theorem~11]{min-grad-est}. In fact, we have $p\gamma<\theta$. Then setting additionally $\theta=n$ and $\gamma=q$ we get easy to verify condition ${n(p-1)}/({n-\gamma})< {m_0p}/({m_0+1}),$ which links the result of \cite{DiC-Pa} to the goals of~\cite{BGV,Cirmi}. See also the studies on  estimates on parabolic problems~\cite{par1,par2} relating to~\eqref{f-Du-mingione}, \eqref{f-Du-giampiero}, respectively.

\medskip

\textbf{Main result. } Let us announce the main accomplishement of this paper. We define $g_m:\rp\to\rp$   by the following formula
\begin{equation}
\label{gm}
g_m(t):=m\circ M^{-1}\circ  {G}(t),
\end{equation}
where $m$ is a doubling and increasing essentially more rapidly than $g$, whereas $M$ is the primitive of $m$. Considering~\eqref{intro:eq:main} under Assumption {\it (Am)}  with the lower-order of a type $b(x,\sigma)\,\sgn\,\sigma\geq c_0 m(|\sigma|)$ (modelled upon \eqref{intro:or-or}), we provide in Theorem~\ref{theo:Lor-Mor} that
\begin{equation}
\label{f-Du-here}\mu\in L^{\theta}(\gamma,q) \implies  g_m(|Du|)\in L^\theta\left( \gamma ,q\right)\quad\text{locally in }\ \Omega
\end{equation}
if $2\leq i_G<\theta\leq n$, $1<\gamma\leq {\theta i_G}/({\theta s_G - \theta +i_G})$, $q\in(0,\infty]$, and $s_m\gamma<(i_m-1)i_G/s_G$. Instead of assumming that $m$ increases essentially faster than $g$ near infinity one can assume a (more restrictive) sufficient condition $s_G-1<i_m$. This result in the case of weighted $p$-Laplace problems~\eqref{intro:eq:plap} with the lower-order term having power-type growth retrieve precisely regularity~\eqref{f-Du-giampiero}. On the other hand, we shall observe the regularizing effect in the general Orlicz case to since~\eqref{f-Du-here} is a higher gradient regularity than~\eqref{f-Du-iwonka}.

 The precise formulation of this result and some special cases are presented and commented  in Section~\ref{sec:results}.  
 
\medskip

\textbf{Interpolation effects. } We shall re-interpret the results of Theorems~\ref{theo:Lorentz-est}, \ref{theo:Morrey-est}, \ref{theo:Bord-Morrey-est}, and~\ref{theo:Lor-Mor} below as interpolation results. Indeed, due to the monotonicity of the operator and the sign condition imposed on $m$, we can decouple information from the equation into two separate inclusions for the second-order term and the solution itself, namely
\[-\dv\, a(x,Du)\in  L^{\theta}(\gamma,q)\qquad \text{and}\qquad  b(x,u)\in L^{\theta}(\gamma,q)\]
and infer the regularity of the gradient of solution as follows: $g_m(|Du|)\in L^{\theta}(\gamma,q)$. It should be stressed that this holds even when the leading part of the operator is dependent on the spacial variable in a just measurable way. The mentioned theorems enable to formulate the related consequences in various general settings. 

The interpolation interpretation is particularly transparent when we consider in~\eqref{intro:eq:main} operator $a(x,Du)$ having linear growth ($|a(\cdot,\xi)\cdot\xi|\approx |\xi|^2$), the lower-order term driven by an Orlicz function $m$, and data from  $L^\gamma$. We can admit also weights as in~\eqref{intro:or-or}: $\omega_1,\omega_2$ being bounded measurable and separated from zero functions. In such a case we can decouple information from the equation and deduce
\[\omega_1(\cdot)|D^{(2)}u|\in L^\gamma\quad \text{and}\quad \omega_2(\cdot) m(u)\in L^\gamma\quad\implies\quad m\circ M^{-1}(|Du|^2)\in L^\gamma.\]
We stress that in the non-weighted case when additionally $m$ is a power function, the above implication retrieves the same weak version of the classical Gagliardo-Nirenberg interpolation result as described in~\cite{DiC-Pa}. On the other hand, in the non-weighted case when $m$ is just doubling we get the same Orlicz interpolation information as can be deduced from~\cite[Theorem~4.3]{kapipa}, cf. Remark~\ref{rem:gn}.

\medskip

\textbf{Methods and challenges. }  We provide estimates for the solution to our main problem~\eqref{intro:eq:main} expressed in the terms of the super-level sets of the maximal operator of the data. Our main inspirations are~\cite{min-grad-est,DiC-Pa,Baroni-Riesz,IC-gradest}.  In fact, the main tool we derive and apply is the following super-level set decay estimate for the maximal operator of~gradient of~solutions, which -- for presenting the intuition -- can be sketched to
\begin{equation}\label{livelli}
\begin{split}
&\left|\left\{ M (|Du|)>K\lambda\right|\right\}  \lesssim \frac{1}{G^{\chi}(K)}\left|\left\{ M (|Du|)> \lambda\right\}\right|+\left|\left\{  {M (|\mu |)}> g_m(\ve\lambda)\right\}\right|,\end{split}
\end{equation}
cf. the definition of the maximal operator  in~\eqref{MDu-Mmu-def} and the full estimate in~\eqref{eq:super-level-est} (Proposition~\ref{prop:super-level-est}). In this estimate $K,\chi$ are certain parameters, $G$ is the function defining the norm in the Orlicz space where the solutions belong to, whereas~$g_m$ is defined in~\eqref{gm}. Recall again that in the $p$-Laplace case related to~\cite{DiC-Pa} we would have here $G(t)\approx|t|^p$ 
 and $g_m(t)\approx |t|^{\frac{m_0p}{m_0+1}}$. Our precise assumptions are collected in the beginning of Section~\ref{sec:results}. 

In the setting of measure data problems, the approach via estimates on level sets of maximal operators as in~\eqref{livelli} has been introduced in~\cite{min07, min-grad-est}. Later developments concern equations with lower order terms~\cite{DiC-Pa}, involve parabolic problems~\cite{BDP,par2}, and operators defined on Orlicz spaces~\cite{IC-gradest}. These estimates also make use of fractional maximal operators and prelude, and are actually linked, to nonlinear potential estimates, first developed for operators with $p$-growth, see~\cite{DM1, DM2,Iw, ku-min-univ, KuMi, KuMi-vec-npt}. As a matter of fact, nonlinear potential estimates in the setting of operators with Orlicz growth have been proven in~\cite{Baroni-Riesz} and this reference provides a lot of useful ideas to our analysis.

\medskip

\textbf{Organization of the paper. } We start the paper giving in Section~\ref{sec:results} the complete set of assumptions and collection of the main results. Afterwards, in Section~\ref{sec:prelim}, we provide Preliminaries that introduce  notation, the Orlicz setting in which our main equation~\eqref{eq:main} is formulated, and information on the notion of solutions we investigate. Section~\ref{sec:aux} is devoted to estimates on solutions to the related homogeneous problem, comparison estimates, and their direct consequences. We derive there also our main tool, i.e. super-level set estimates for the maximal operator of the gradient. The final proofs of the main theorems listed above are given in Section~\ref{sec:proofs}. In the end, in Appendix, we give necessary definitions, concise information on~the involved functional spaces, and  basic and classical estimates.

\section{The results}\label{sec:results}

The problem we consider generalizes the problem~\eqref{intro:eq:plap} towards nonstandard growth described in the Orlicz setting. 

Before we collect the set of assumptions and present below the main theorems we shall give some basic remarks on the notation. In the following we denote by $c$ a constant that may vary from line to line. Sometimes to skip rewriting a constant, we use $\lesssim$. By $a\approx b$, we mean $a\lesssim b$ and $b\lesssim a$. By $B_R$ we shall denote a ball usually skipping prescribing its center, when it is not important. Then by $cB_R=B_{cR}$ we mean then a ball with the same center as $B_R$, but with rescaled radius $cR$.

We say that a function $h\in C^1(0,\infty)$ is  doubling (and denote it by $h\in\Delta_2\cap\nabla_2$), when 
\begin{equation}\label{doubling} 1\leq i_h=\inf_{t>0}\frac{th'(t)}{h(t)}\leq \sup_{t>0}\frac{th'(t)}{h(t)}=s_h<\infty,
\end{equation}
which in particular implies $h,\wt{h}\in\Delta_2$. The parameters $i_h$ and $s_h$ describe the speed of growth of $h$. More information on such functions can be found in Section~\ref{sec:Or}.

An increasing convex function $h_1$ is said to {increase essentially faster} than $h_2$ near infinity, if $\lim_{t\to +\infty} {h_1^{-1}(t)}/{h_2^{-1}(t)}= 0 $. In the case of power functions it suffices to compare the exponents.

\medskip

Let us present the set of assumptions we engage.

\noindent\textbf{Assumption {(Am)}}

Recall $\Omega$ is a bounded domain in $\rn$, $n\geq 2$. Suppose $g,m\in C^1(0,\infty)$ are doubling convex functions such that $m(0)=0=g(0)$ and $m$ increases essentially faster than $g$ near infinity. The primitives of~$g$ and $m$ are $G,M\in C^2(0,\infty)$, respectively, i.e. $G'(t)=g(t)$ and $M'(t)=m(t)$.  Let indexes $i_G,s_G$ and $i_m,s_m$ describe their growth as in~\eqref{doubling}. Assume further that $\gamma>1$ is such that \[ s_m\gamma<(i_m+1)\frac{i_G}{s_G}.\]  We investigate the problem 
\begin{equation}\label{eq:main}
-{\dv}\, a(x,Du)+b(x,u)=\mu\quad\text{in}\quad \Omega,
\end{equation}
where $\Omega\subset\rn$, $\mu$ is  a Borel measure with finite total mass $|\mu|(\Omega)<\infty$, whereas $a:\Omega\times\rn\to\rn$ is a~Carath\'{e}odory function (measurable with respect to the first variable and continuous with respect to the second one) that satisfies
\begin{equation*}
\left\{\begin{array}{l}
\langle a(x,\xi_1)-a(x,\xi_2), \xi_1-\xi_2\rangle\geq \nu\frac{g(|\xi_1|+|\xi_2|)}{|\xi_1|+|\xi_2|}|\xi_1-\xi_2|^2,\\
|a(x,z)| \leq Lg(|z|).
\end{array}\right.
\end{equation*} 
Moreover, we assume that the lower order-term $b:\Omega\times\rp\to\r$ satisfies the sign condition, that is that there exists $c_0>0$ such that\[b(x,\sigma)\,\sgn\,\sigma\geq c_0\, m(|\sigma|) \]
and for every $\tau>0$ function $\sup_{|\sigma|\leq\tau}|b(x,\sigma)|\in L^1_{loc}(\Omega)$.

\begin{rem}\rm For the transparence of the reasoning, we shall present the proofs for $b(x,\sigma)=m(|\sigma|)\,\sgn\,\sigma$.
\end{rem}

\bigskip

\noindent\textbf{Main results }

 Let us present the precise formulations of our main accomplishements. We prove estimates of gradient integrability in several function spaces. See Subsection~\ref{ssec:fn-sp} (in Appendix) for necessary definitions and notation e.g. of the function spaces or the averaged norms. The main proofs are provided in Section~\ref{ssec:main-proofs}.\\  Recall that $g_m$ is defined in~\eqref{gm}.

\begin{theo}[Lorentz regularity]\label{theo:Lorentz-est}
 Let $u\in W^{1,1}_{loc}(\Omega)$ be a local SOLA to~\eqref{eq:main} with $G,g$ satisfying {\it (Am)} and 
 \begin{equation}\label{q-Lorentz}
 1< \gamma \leq \frac{n i_G}{ns_G-n+i_G} \qquad\text{and}\qquad 0<q\leq \infty.
 \end{equation}
If $f\in L(\gamma,q)$ locally in $\Omega,$ then $g_m(|Du|)\in L\left(\gamma,q\right)$ locally in $\Omega$. Moreover, there exists $c=c(n,i_G, s_G,q,s)$, such that for every ball $B_R\subset\subset\Omega$ we have
 \begin{equation}
\label{eq:Lorentz-est}  \avenorm{ g_m( |Du|) }_{L\left(\gamma,q\right)(B_{R/2})}\leq c\,g_m\left( \barint_{B_{2R} }|Du|dx\right) +c \avenorm{\mu}_{L(\gamma,q )(B_R)}.
\end{equation} 
\end{theo}

\begin{theo}[Morrey regularity]\label{theo:Morrey-est}
 Let $u\in W^{1,1}_{loc}(\Omega)$ be a local SOLA to~\eqref{eq:main} with $G,g$ satisfying {\it (Am)} and
\begin{equation}
\label{q-Morrey}i_G\leq \theta\leq n\qquad\text{and}\qquad 1<\gamma\leq\frac{\theta i_G }{\theta s_G-\theta+i_G }.
\end{equation} 
If  $\mu\in L^{\gamma,\theta} $ locally in $\Omega$, then
 $g_m(|Du|)\in L^{\gamma,\theta}$ {locally in }$\Omega.$ 
 Moreover, there exists $c=c(n,G,q,\gamma)$, such that for every ball $B_R\subset\subset\Omega$ we have
 \begin{equation*}
  \avenorm{ g_m( |Du|) }_{L^{\gamma,\theta}(B_{R/2})}\leq c\,R^{\frac{\theta}{\gamma} }\, g_m\left( \barint_{B_{2R} }|Du|\,dx\right) +c \avenorm{\mu}_{L^{\gamma,\theta}(B_R)}.
\end{equation*} 
\end{theo}
We get the following extension of \cite[Theorem~6]{DiC-Pa}. 
\begin{theo}[Borderline Morrey case]\label{theo:Bord-Morrey-est}
 Let $u\in W^{1,1}_{loc}(\Omega)$ be a local SOLA to~\eqref{eq:main} with $G,g$ satisfying {\it (Am)} and parametrs as in~\eqref{q-Morrey}. If $i_G\leq\theta\leq n$ and $\mu\in L\log L^\theta $ locally in $\Omega$, then $ g_m(|Du|)\in L^{1,\theta}$ locally in $\Omega$.  Moreover, there exists $c=c(n,G,q,\gamma)$, such that for every ball $B_R\subset\subset\Omega$ we have
 \[\avenorm{ g_m( |Du|) }_{L^{1,\theta}(B_{R/2})}\leq c\,R^{ \theta } g_m\left( \barint_{B_{2R} }|Du|dx\right) +c \avenorm{ \mu }_{L\log L^\theta(B_R)}.
\]
\end{theo}

Now we present the main result of the paper extending \cite[Theorem~1]{DiC-Pa} and retrieving to it in the $p$-Laplace case. Note that the range of the parameters includes the upperbound of the rage of parameters $\theta$ and $\gamma$, as well as the Marcinkiewicz case ($q=\infty$).

\begin{theo}[Lorentz-Morrey regularity]\label{theo:Lor-Mor}
 Let $u\in W^{1,1}_{loc}(\Omega)$ be a local SOLA to~\eqref{eq:main} with $G,g$ satisfying {\it (Am)}, parametrs $\theta,q$ be as in~\eqref{q-Morrey} and $q\in(0,\infty]$. If  $f\in L^{\theta}(\gamma,q) $ locally in $\Omega$, then $g_m(|Du|)\in L^\theta\left(\gamma,q \right)$ {locally in }$ \Omega.$  Moreover, there exists $c=c(n,G,\gamma,q)$, such that for every ball $B_R\subset\subset\Omega$ we have
\[\avenorm{g_m( |Du |)}_{L^\theta\left(\gamma,q\right)(B_{R/2})} \leq cR^{\frac{\theta}{\gamma}} g_m\left(\barint_{B_{2R}}|Du |dx\right)+c\avenorm{\mu}_{L^\theta(\gamma,q)(B_R)}.\]
\end{theo}

\section{Preliminaries}\label{sec:prelim}

\subsection{The Orlicz setting}\label{sec:Or}
We study the solutions to PDEs in the Orlicz-Sobolev spaces equipped with a modular function $B$ - an increasing and convex function satisfying doubling condition~\eqref{doubling}.

\begin{defi}[Orlicz-Sobolev space]\label{def:OrSob:sp} 

 By the Orlicz space ${L}_B(\Omega)$  we understand the space of measurable functions endowed with the Luxemburg norm 
\[||f||_{L_B}=\inf\Big\{\lambda>0:\ \ \int_\Omega B\left( \tfrac{1}{\lambda}|f(x)|\right)\,dx\leq 1\Big\}.\]
 We define the Orlicz-Sobolev space  $W^{1,B}(\Omega)$  as $
W^{1,B}(\Omega)=\big\{f\in L_B(\Omega):\ \ |D f|\in L_B(\Omega)\big\},$ endowed with the norm
\[
\|f\|_{W^{1,B}(\Omega)}=\inf\Big\{\lambda>0 :\ \   \int_\Omega B\big( \tfrac{1}{\lambda}|f(x)|\big)dx+\int_\Omega B\big( \tfrac{1}{\lambda}|Df(x)|\big)dx\leq 1\Big\} 
\]
and  by $W_0^{1,B}(\Omega)$ we denote a closure of $C_c^\infty(\Omega)$ under the above norm. 
\end{defi} 

In the functional analysis of the Orlicz setting an important role is played by  $\wt{B}$ -- the defined below  complementary~function (called also the Young conjugate, or the Legendre transform) to a function  $B:\r\to\r$. The complementary~function is given by the following formula $\wt{B}(t)=\sup_{s>0}(s\cdot t-B(s)).$ Note that if $\wt{B}$ is a~complementary function to a Young function $B$, then  $\wt{B}$ is also a Young function. Moreover, it is complementary in the sense that \begin{equation}
\label{eq:formagic}
 t\leq B^{-1}(t)\wt{B}^{-1}(t)\leq 2t.\end{equation} 

\noindent\textbf{Growth. }  We would like to comment the growth condition under which we work. The typical assumption on the growth within the Orlicz classes is the following one.

\begin{defi}[$\Delta_2$-condition]\label{def:D2}
 We say that a function $B:\r\to\r$ satisfies $\Delta_2$-condition if there exists a constant $c_{\Delta_2}>0$ such that $B(2s)\leq c_{\Delta_2}B(s).$
\end{defi} 

It describes the speed and regularity of the growth. For instance when $B(s) = (1+|s|)\log(1+|s|)-|s|$, its complementary function is  $\widetilde{B}(s)= \exp(|s|)-|s|-1$. Then $B\in\Delta_2$ and   $\widetilde{B}\not\in\Delta_2$.

 We point out that for $C^1$ convex functions $B$ the condition~\eqref{doubling}  is equivalent to $B,\wt{B}\in\Delta_2$,~\cite[Section~2.3, Theorem~3]{rao-ren}, which here we call the doubling growth. Indeed, if $s_B<\infty$ then $B\in\Delta_2$, whereas $i_B>1$ entails the $\Delta_2$-condition imposed on $\wt{B}$. It also implies a to comparison with power-type functions in the sense that when $B$ satisfies~\eqref{doubling}, then
\begin{equation}
\label{B-power-compar}
\frac{B(t)}{t^{i_B}}\quad\text{is non-decreasing}\qquad\text{and}\qquad\frac{B(t)}{t^{s_B}}\quad\text{is non-increasing}.
\end{equation}
 On the other hand, condition~\eqref{doubling} imposes also non-oscillatory behaviour between types of growth, see e.g. \cite[Example~3.2]{CGZG}.

\begin{rem}[\cite{adams-fournier}] Since condition~\eqref{doubling} imposed on $B$ implies $B,\wt{B}\in\Delta_2$, the Orlicz-Sobolev space $W^{1,B}(\Omega)$ we deal with is separable and reflexive. Moreover, the duality is given by $\big(W^{1,B}(\Omega)\big)^*\sim W^{1,\wt{B}}(\Omega)$.
\end{rem}
 
\begin{lem}[\cite{rao-ren}]\label{lem:D2} If $B\in\Delta_2$, then $B(r+s)\lesssim B(r)+B(s)$.
\end{lem}

\begin{rem}\rm We notice that since $g$ satisfies~\eqref{doubling}, due to~\cite{DieEtt}, we have
\begin{equation}
\label{G<g}
G(|\xi_1-\xi_2|)\lesssim \frac{g(|\xi_1-\xi_2|)}{|\xi_1-\xi_2|}|\xi_1-\xi_2|^2\lesssim \frac{g(|\xi_1|+|\xi_2|)}{|\xi_1|+|\xi_2|}|\xi_1-\xi_2|^2.
\end{equation}
\end{rem}

\noindent\textbf{Embeddings. } For Sobolev--Orlicz spaces expected embedding theorems hold true. Suppose $n>1$. We distinguish two possible behaviours of $B$
\begin{equation}
\label{intB}
\int^\infty\left(\frac{t}{B(t)}\right)^\frac{1}{n-1}dt=\infty \qquad\text{and}\qquad
\int^\infty\left(\frac{t}{B(t)}\right)^\frac{1}{n-1}dt<\infty,
\end{equation}
which roughly speaking describe slow and fast growth of $B$ in infinity, respectively. The condition imposing slow growth of $B$, namely \eqref{intB}$_1 $, corresponds to  the case of $p$-growth with $p\leq n$. Then we expect $W_0^{1,B}\hookrightarrow{} L_{\widehat{B}}$ with $\widehat{B}$ growing faster than $B$ (it is presented below). In the case of quickly growing modular function, i.e. when~\eqref{intB}$_2$ holds (corresponding to $p>n$), it holds that $W_0^{1,B}\hookrightarrow{}L^\infty$. Below we give details.  

We apply the optimal embeddings due to~\cite{Ci96-emb}, where  the Sobolev inequality is proven under the restriction 
\begin{equation}\label{int0B}\int_0\left(\frac{t}{B(t)}\right)^\frac{1}{n-1}dt<\infty, 
\end{equation} 
concerning the growth of $B$ in the origin. Nonetheless, the properties of $L_B$ depend on the behaviour of $B(s)$ for large values of $s$ and~\eqref{int0B} can be easily by-passed in application. When we consider 
\begin{equation}\label{BN}
A_n(s)=\left(\int_0^s\left(\frac{t}{B(t)}\right)^\frac{1}{n-1}dt\right)^\frac{n-1}{n} \qquad \text{and}\qquad
B_n(t)=B(A_n^{-1}(t)), 
\end{equation}
the following result holds true.

\begin{theo}[Sobolev embedding, \cite{Ci96-emb}]\label{theo:Sob-emb} Let $\Omega\subset\rn$, $n>1$, be a bounded open set.
\begin{itemize}
\item[(slow)] If $B$ is a Young function satisfying \eqref{int0B} and \eqref{intB}$_1 $, then there exists a constant $c_s=c_s(n)$, such that for every $u\in W_0^{1,B}(\Omega)$ it holds that \[\int_\Omega B_n\left(\frac{|u|}{c_s\big(\int_\Omega B(|Du|)dx\big)^\frac{1}{n}}\right)dx\leq \int_\Omega B(|Du|)dx.\]

\item[(fast)]  If $B$ is a Young function satisfying \eqref{intB}$_2 $, then then there exists a constant $c(n)$, such that for every $u\in W_0^{1,B}(\Omega)$ it holds that $\|u\|_{L^\infty(\Omega)}\leq \|Du\|_{L_B(\Omega)}.$
\end{itemize}
\end{theo}

\subsection{Notion of SOLA and its existence}\label{ssec:SOLA}

 Investigating the general elliptic Dirichlet problem 
\begin{equation} \label{intro:ell:f}
- \dv \,a(x,D u)+b(x,u)= f\quad\text{or}\quad\mu
\end{equation} 
involving $a$~from an Orlicz class and on the right--hand side data merely integrable or in the space of~measures, one should consider a special notion of solutions. Indeed, the weak formulation of~\eqref{intro:ell:f}, i.e.
\[\int_\Omega a(x,D u)D\vp\,dx+\int_\Omega b(x,u)\vp\,dx=\int_\Omega  \vp\,d\mu,\]
is expected to hold for every $\vp$ in the Orlicz-Sobolev space $W_0^{1,G}(\Omega)$. It is not possible in general when the data is just arbitrary. Of course one can consider the distributional solutions, but they can be wild, cf. the classical example~\cite{Serrin-pat}. Nonetheless, it is possible to define solutions which in classical cases reproduce weak solutions, while for measure data enjoy some good regularity properties. There are at least three different classical approaches to this problem. The notion of renormalized solutions  appeared first in~\cite{diperna-lions}, whereas the entropy solutions come from~\cite{bbggpv,dall}. We investigate the SOLA studied starting from~\cite{bgSOLA-jfa,bgSOLA-cpde}. Let us mention that all the results involving SOLA naturally concerns only $p>2-1/n$, since it is necessary to ensure that $u\in W^{1,1}_{loc}(\Omega)$ for arbitrary measure data.   See~\cite{IC-pocket} for a survey on problems in the generalized setting with data below duality in various nonstandard growth settings.

To consider the datum $f$ not belonging to the dual space, we adopt the notion of SOLA. Since under certain restrictions the mentioned notions coincide~\cite{Rak,KiKuTu}, it suggests that the gradient estimates we obtain for SOLA can be shared by the other types of solutions as well.

\begin{defi}[Local SOLA]\label{def:SOLA}
A function $u\in W^{1,1}_{loc}
(\Omega)$ is called a local {SOLA} to~\eqref{intro:eq:main} if  problems \begin{equation}
\label{eq:main-k-for-SOLA}
- \dv\, a(x, Du_k ) +b(x,u_k)= f_k=\mu_k\in L^\infty(\Omega)
\end{equation} with $\mu_k\to\mu\in{\cal M}(\Omega)$ $*$--(locally)--weakly in the sense of measures, that is $
\lim_{k\to\infty}\int_\Omega \vp\,f_k\,dx=\int_\Omega\vp\,d\mu$ for every  continuous function $\vp$ with compact support in $\Omega$, and satisfying
$\limsup_{k\to\infty} |\mu_k |(B) \leq |\mu|(B)$
for every ball $B\subset\Omega$  
have solutions  $\{u_k\}_k\subset W^{1,G}_{loc}
(\Omega)$ such that \begin{flalign*} u_k\xrightarrow[k\to\infty]{} u\quad&\text{ strongly\ \ in\ \  }W^{1,1}_{loc}
(\Omega),\\
b(\cdot,u_k)\xrightarrow[k\to\infty]{}b(\cdot,u)\qquad\text{and}\qquad& g(|D u_k|)\xrightarrow[k\to\infty]{}g(|D u|)\quad\text{strongly\ in \  }L^{1}_{loc}(\Omega).\end{flalign*}
\end{defi}
The existence of such solutions can be inferred with no substantial difficulties provided $m$ increases essentially faster than $g$ near infinity. This assumption ensures that $|a(x,Du)|\in L^1_{loc}(\Omega)$. For this one can either on the basis of~\cite{MT} follows the arguments of~\cite{CGZG,DiC-Pa}, or using considerations from~\cite[Section 7]{Baroni-Riesz} and~\cite{CiMa,bgSOLA-jfa,DiC-Pa}. Analogical result on existence of  renormalized solutions for problems with lower-order terms and $L^1$-data (posed in Musielak-Orlicz space) is provided in~\cite{gwiazda-ren-ell}.

In the case of SOLA by uniqueness we mean independence of the limit solution of the choice of a sequence of approximate data $(f_k)$ and, consequently, of the sequence of related approximate solutions. Note that the uniqueness is expected within this notion of solutions if the data $\mu$ can be expressed by a locally integrable function, while for general measure data it is an open problem. 

\section{Auxiliary results}\label{sec:aux}
In order to compare the properties of solutions to our main equation to the solutions to the homogeneous equation (i.e. null-data one) we start with some integrability results for solutions to homogeneous problem itself and then we infer comparison estimates.

\subsection{Homogeneous problem}\label{ssec:homog}

This subsection is devoted to various estimates for $v$ solving the homogeneous problem
\begin{equation}\label{eq:homog} 
-{\dv}\, a(x,Dv)+m(v)\tfrac{v}{|v|}=0.
\end{equation} 
Recall that we expect more regularity than in the lower-term free case of~\cite{IC-gradest}.

\begin{prop}[Estimates for the homogeneous problem]\label{prop:homo-problem} Suppose $B_{2R}\subset\subset A\subset\rn$, $A$ is a bounded set, and $v \in W^{1,G} (Am)$ is a weak solution to~\eqref{eq:homog} on $A$, where $a:\rn\to\rn$ and $G,g:[0,\infty)\to[0,\infty)$ satisfy Assumption~(Am). Then
\begin{itemize}

\item[(i)]  there exists a constant $c=c(n,\nu,L,s_G)$, such that\begin{equation}
\label{rev-Hold}
\barint_{B_{R }}G(|Dv|)\, d x \leq c\,G\left(\barint_{B_{2R} }   
|Dv|\, d x \right),
\end{equation}

\item[(ii)] then there exist $ {c}_1,{c}_2>0$ and $\chi>1$, such that
\begin{equation}\label{higher-int}\barint_{B_R} G^{\chi}(|D v|)\,dx \leq  {c}_1\, G^{\chi}\left(\barint_{B_{2R}}   |D v|  \,dx\right)+ {c}_2,\end{equation}

\item[(iii)]there exists $c>0$, such that  
\begin{equation}
\label{inq:cacc} 
\int_{B_{ R}}G(|Dv|)\,dx\leq c \int_{B_{2R}}G\left(\frac{|v-(v)_{B_R}|}{R}\right)
dx+ R^n,
\end{equation}

\item[(iv)] for $\vr<R $,  there exist $c,\beta_m>0$
, such that 

\begin{equation}
\label{inq:Morrey-continuity-m} \barint_{B_{\vr }}g_m(|Dv|)dx\leq c\left(\frac{\vr}{R}\right)^{-\beta_m} \barint_{B_{2R}}g_m (|Dv|)dx.\end{equation}
\end{itemize}
\end{prop} 
Note that due to the convexity and the sign property of $m$ the proof of {\it (i)-(iii)} is essentially the same as in~\cite[Proposition~4.1]{IC-gradest}, since the term involving $m$ can be just dropped. Compare with the proof of~Lemma~\ref{lem:comp-sol}. For  {\it (iv)} one has to additionally observe that $m\circ M^{-1}$ is concave. \\
For this proof there is no gain from restricting the form of the lower-order term. The result will be however applied only in the above form.

\subsection{Comparison estimates and direct consequences}
We provide a comparison estimate between solution to~\eqref{eq:main} and $v\in u+W^{1,G}_0(B_R)$ solving
\begin{equation}\label{eq:comp-map}
\left\{\begin{array}{ll}
-{\dv}\, a(x,Dv)+m(v)\frac{v}{|v|}=0&\text{ in }B_{R},\\
v=u&\text{ on }\partial B_{R}.
\end{array}\right.
\end{equation} For existence and uniqueness for this problem we refer to~\cite[Lemma 5.2]{Lieb91}.

Let us first compare the integrability of the difference of solutions and then we infer the comparison estimate for their gradients.

\begin{lem}\label{lem:comp-sol} Suppose $\gamma\geq 1$ is an arbirary number, $m\in \Delta_2\cap \nabla_2$, and $A=\avenorm{ m^{-1}(|u-v|)}_{L^\gamma(B_R)}$. Then
\[\avenorm{ m(A+|u-v|)}_{L^\gamma(B_R)}\leq \avenorm{\mu}_{L^\gamma(B_R)}.\]
\end{lem}
\begin{proof}
We start with the case of $\gamma=1$. Subtracting weak formulations of~\eqref{eq:main} and~\eqref{eq:comp-map} we get
\begin{equation}
\label{diff-weak}
\int_{B_R}\langle a(x,Du)-a(x,Dv), D\phi\rangle\,dx+
\int_{B_R}\left( m(|u|)\tfrac{u}{|u|}- m(|v|)\tfrac{v}{|v|} \right)\phi\,dx=\int_{B_R} \phi\,d\mu.
\end{equation}
We use as a test function $\phi=\Phi_\ve(u-v)$ tending to $\sgn(u-v)$ when $\ve\to 0$. Then, dropping the first nonnegative term and using Fatou's Lemma we get
\[\int_{B_R}\left( m(|u|)\tfrac{u}{|u|}- m(|v|)\tfrac{v}{|v|} \right)\sgn(u-v)\,dx\leq \int_{B_R} |f|\,dx.\]
Therefore, due to monotonicity of $r\mapsto m(r)/r$ and the fact that $m\in \Delta_2\cap \nabla_2$, we can deduce that
\[\int_{B_R} m (|u-v|)\,dx\leq \int_{B_R}\,d|\mu|.\]
To get the final claim we need to apply Jensen's inequality and doubling condition.

When $\gamma>1$, we use as a test function $\phi=m^{\gamma-1}(|u-v|)\Phi_\ve(u-v)$. By the same arguments as above combined with H\"older's inequality we get
\[\int_{B_R} m^\gamma(|u-v|)\,dx\leq
 \int_{B_R} m^{\gamma-1}(|u-v|)\,d|\mu|\leq \|\mu\|_{L^\gamma(B_R)} \left( \int_{B_R} m^\gamma(|u-v|)\,dx\right)^\frac{\gamma-1}{\gamma},\]
 which after rearranging terms becomes 
$\|m(|u-v|)\|_{L^\gamma(B_R)}\leq \|\mu\|_{L^\gamma(B_R)}.$ Again to conclude it suffices to use Jensen's inequality  and doubling condition.
\end{proof}
 
We have also the following comparison of the growth of the involved convex functions.
\begin{lem}\label{lem:somekindofmagic} Having $g_m$ from~\eqref{gm} satisfying Assumption (Am),  we have for $D,s>0$ and $\gamma\geq 1$ that
\[g_m^\gamma(D) \lesssim G(D)\frac{s}{m^{\gamma-1}(s)}+ m^\gamma(s).\]
The implicit constant in this estimate depends only on $i_m,s_m,i_G,s_G$.
\end{lem}
\begin{proof}
We shall first observe that, when we denote a convex doubling function $H_{m,\gamma}(t)=\big(m^\gamma\circ M^{-1}\big)^{-1}(t)$, then directly from the formula for $g_m$~\eqref{gm}, \eqref{eq:formagic}, and  doubling properties of $m$, we can infer that for any $s>0$ that
\[m^\gamma\circ M^{-1}\circ\wt{H_{m,\gamma}}\left(\frac{s}{m^{\gamma-1}(s)}\right)\approx m^\gamma\circ M^{-1}\circ\wt{H_{m,\gamma}}\left(\frac{M(s)}{m^\gamma(s)}\right)\approx m^\gamma(s).\]

Applying Young's inequality with $H_{m,\gamma}$ and its conjugate, and then the above observation, we arrive at
\begin{flalign*}
g_m^\gamma(D)&= m^\gamma\circ M^{-1}\left(\frac{G(D)m^{\gamma-1}(s)}{s}\cdot\frac{s}{m^{\gamma-1}(s)}\right)\leq m^\gamma\circ M^{-1}\left(H_{m,\gamma}\left[\frac{G(D)m^{\gamma-1}(s)}{s}\right]+\wt{H_{m,\gamma}}\left[\frac{s}{m^{\gamma-1}(s)}\right]\right)\\
&\leq m^\gamma\circ M^{-1}\left(2H_{m,\gamma}\left[\frac{G(D)m^{\gamma-1}(s)}{s}\right]\right)+m^\gamma\circ M^{-1}\left(2\wt{H_{m,\gamma}}\left[\frac{s}{m^{\gamma-1}(s)} \right]\right) \lesssim G(D) \frac{m^{\gamma-1}(s)}{s} + m^\gamma (s).
\end{flalign*} 
\end{proof}

We compare now the integrability of the difference of the gradients of solutions. 
 
\begin{prop} \label{prop:comp-B-xi}
Suppose $\gamma\geq 1$ is an arbitrary number, $a:\Omega\times\rn\to\rn$ satisfies Assumption (Am), and $\avenorm{\mu}_{L^\gamma(B_R)}<\infty$. If $u\in W^{1,G}(\Omega)$ is a local SOLA to~\eqref{eq:main} and $v\in u+W^{1,G}_0(B_R)$ is a weak solution to~\eqref{eq:comp-map} on $B_R$, then   there exist a constant $c=c(n,\nu,m,G)>0$, such that
\begin{equation}\label{eq:comp-est}
\avenorm{g_m(|Du-Dv|)}_{L^\gamma(B_R)}\leq c\,\avenorm{\mu}_{L^\gamma(B_R)}.
\end{equation}
\end{prop}

\begin{proof} In order to get~\eqref{eq:comp-est}, we rescale the equation to the unit ball, provide an estimate and then rescaling back we arrive at the claim.

\medskip

\noindent\textsc{Step 1. Rescaling. }

 Let us note that if $\mu(B_R)=0,$ then monotonicity of the vector field $a$ implies that $u=v$ and~\eqref{eq:comp-est} trivially follows. Otherwise, i.e. when $\mu(B_R)\neq 0,$ we rescale the equation. For this we denote  
\begin{equation*}
\cmf=m^{-1}\big(\avenorm{\mu}_{L^\gamma(B_R)}\big).
\end{equation*}
and moreover we set
 \begin{equation}
\label{resc-comp} 
 \begin{array}{llll}
\bu(x)=&\frac{u(x_0+Rx)}{\cmf},&\qquad 
\bv(x)=&\frac{v(x_0+Rx)}{\cmf},\\\\
\ba(x,z)=&\tfrac{1}{R\avenorm{\mu}_{L^\gamma(B_R)}} {a\left(x_0+Rx,\tfrac{\cmf}{R}z\right)},&\qquad 
\bamu(x)=&\frac{\mu(x_0+Rx)}{\avenorm{\mu}_{L^\gamma(B_R)}},\\\\
\bm(s)=&\tfrac {1}{\avenorm{\mu}_{L^\gamma(B_R)}} m(\cmf s),&\qquad \bar{M}(s)=&\tfrac{1}{\cmf\avenorm{\mu}_{L^\gamma(B_R)}}M(\cmf s).
 \end{array} 
 \end{equation}
Then $-\dv\,\ba(x,D\bu)+\bm(|\bu|)\tfrac{\bu}{|\bu|}=\bmu$ in $B_1$, $\avenorm{\bmu}_{L^\gamma(B_1)}\leq 1$ and the growth and coercivity of $\ba$ are governed by \[\bar{g}(s)=\tfrac{1}{\cmf\avenorm{\mu}_{L^\gamma(B_R)}}g\left(\tfrac{\cmf}{R} s\right)\qquad\text{ with }\qquad \bar{\nu}=\nu\quad \text{ and }\quad  \bar{L}=\tfrac{\cmf}{R  } {L} .\]
Indeed, 
we have
\[\begin{split}
\langle \ba(x,z_1)-\ba(x,z_2),z_1-z_2\rangle&=\tfrac{1}{R \avenorm{\mu}_{L^\gamma(B_R)}} \langle a\left(x_0+Rx,
\tfrac{\cmf}{R}z_1\right)-a\left(x_0+Rx,\tfrac{\cmf}{R} z_2\right),z_1-z_2\rangle\\
&\geq \tfrac{1}{R \avenorm{\mu}_{L^\gamma(B_R)}} \tfrac{R}{\cmf}\langle  a\left(x_0+Rx,\tfrac{\cmf}{R}z_1\right)-a\left(x_0+Rx,
\tfrac{\cmf}{R} z_2\right),
\tfrac{\cmf}{R}z_1-\tfrac{\cmf}{R}z_2\rangle\\
&\geq \bar{\nu}\tfrac{\bar g(|z_1|+|z_2|)}{|z_1|+|z_2|}|z_1-z_2|^2
\end{split}\] 
and
\[|\ba(x,z)|=\tfrac{1}{R \avenorm{\mu}_{L^\gamma(B_R)}}\left| a\left(x_0+Rx,\tfrac{\cmf}{R}z\right)\right|\leq \tfrac{L}{R \avenorm{\mu}_{L^\gamma(B_R)}}g\left(\tfrac{\cmf}{R} |z|\right)=\bar{L}\bg(|z|).\]

Then~\eqref{eq:main} and~\eqref{eq:comp-map} implies
\begin{equation*}
 -\dv\big( \ba(x,D\bu)-\ba(x,D \bv)\big)+\bm(\bu)\tfrac{\bu}{|\bu|}-m(\bv)\tfrac{\bv}{|\bv|}=\bamu\qquad\text{in }\ B_1 
\end{equation*}
admitting the weak formulation
\begin{equation}\label{eq:weak-diff-comp-est}
 \int_{B_1}\langle \ba(x,D\bu)-\ba(x,D \bv),D\vp\rangle\,dx+\int_{B_1}\left(\bm(\bu)\tfrac{\bu}{|\bu|}-m(\bv)\tfrac{\bv}{|\bv|}\right)\vp\,dx=\int_{B_1}\vp\, d\bamu.
\end{equation}
Our aim now is to provide estimate
\begin{equation}\label{eq:resc-comp-est}
 \int_{B_1} \bg^\gamma  (|D\bu-D \bv|)\,dx \leq c,
\end{equation}
where the rescaled function $\bg_m$ is given $\bg_m(s)=\bm \circ\bar{M}^{-1}\circ\bar{G}(s).$

\newpage
\noindent\textsc{Step 2. Measure data estimates. } 



Recall that we consider the case of slowly growing $G$, i.e. `slowly' in the sense of Sobolev's embedding (Theorem~\ref{theo:Sob-emb}). Without loss of the generality we can assume that and \[A:=\avenorm{m^{-1}(|\bu-\bv|)}_{L^\gamma(B_1)}>0,\] because otherwise we would deal with $u=v$ a.e. and in turn $u$ would share stronger regularity than we infer here. According to Lemma~\ref{lem:somekindofmagic} we have
\begin{flalign}\label{1st-est}
\barint_{B_1} \bg_m^\gamma(|D\bu-D \bv|)\,dx \leq c\ \barint_{B_1}\bar{G}(|D\bu-D \bv|)\frac{\bm^{\gamma-1}(A+|\bu-\bv|)}{ A+|\bu-\bv| }dx+c\ \barint_{B_1} \bm^\gamma(A+|\bu-\bv|)\,dx.
\end{flalign}
  For any $k>0$ and $\sigma\in\r$ let us denote
\[T_k(\sigma)=\max\{-k,\min\{k,\sigma\}\}\quad\text{and}\quad \Phi_k(\sigma)=T_1(\sigma-T_k(\sigma)).\]
At first we use
\[\vp=T_k\left( {\bu-\bv} \right)\in W^{1,\bG}_0(B_1)\cap L^\infty(B_1)\]
as a test function in~\eqref{eq:weak-diff-comp-est}. Note that
\[D\vp= {D(\bar u-\bar v)} \mathds{1}_{C_k},\qquad\text{where}\qquad C_k=\left\{x\in B_1:\  {|\bar u(x)-\bar v(x)|} \leq k\right\}.\]
Notice that since $\bg$ satisfies~\eqref{doubling} we have~\eqref{G<g} and, consequently,
\[c\, \nu\int_{C_k}\bG(|D\bu-D\bv|)dx\leq \int_{C_k}\langle \ba(x,D\bu)-\ba(x,D \bv),D\bu-D\bv\rangle\,dx=\int_{B_1}\langle \ba(x,D\bu)-\ba(x,D \bv),D\vp\rangle\,dx.\]

On the other hand, when we recall that $\avenorm{\bmu}_{L^1(B_1)}\leq 1$, we observe
\[\left|\int_{B_1}T_k\left( {\bar u-\bar v} \right)\, d\bmu\right|
\leq  \int_{B_1} k\, d|\bar\mu | =k |\bar \mu |(B_1)=k.\]
 
Making use of the sign condition on $m$ we obtain
\begin{equation}
\label{G-Ck} \nu\int_{C_k}\bG(|D\bu-D\bv|)dx\leq c\int_{B_1}\langle \ba(x,D\bu)-\ba(x,D \bv),D\vp\rangle\,dx\leq c\int_{B_1}\vp\, d\bmu\leq ck.\end{equation}

Using the same arguments with $\vp=\Phi_k( u-v )\in W^{1,\bar G}_0(B_1)\cap L^\infty(B_1)$ as a test function in~\eqref{eq:weak-diff-comp-est} we get
\begin{equation}
\label{G-Ck+1-Ck} \nu \int_{C_{k+1}\setminus C_k}\bG(|D\bu-D\bv|)dx\leq c\int_{B_1}\langle \ba(x,D\bu)-\ba(x,D \bv),D\vp\rangle\,dx\leq c\int_{B_1}\vp\, d\bmu\leq c.\end{equation}

We go back to~\eqref{1st-est} and, according to doubling properties of $m$ and relation of $\gamma$ to other parameters from {\it Assumption (Am)}, we estimate
\begin{flalign*}
 \int_{B_1}\bar{G}(|D\bu-D \bv|) \frac{\bm^{\gamma-1}(A+|\bu-\bv|)}{(A+|\bu-\bv|)}dx&=
\sum_{j=0}^\infty \int_{C_{k+1}\setminus C_k}{\bar{G}(|D\bu-D \bv|)}\frac{\bm^{\gamma-1}(A+|\bu-\bv|)}{ A+|\bu-\bv| }dx\\
&\leq \frac{c}{\nu}
\sum_{k=0}^\infty \int_{C_{k+1}\setminus C_k} \sum_{j=0}^k \frac{\bm^{\gamma-1}(A+j)}{A+j} \ d|\bmu|\\
&\leq \frac{c}{\nu}
\sum_{k=0}^\infty \int_{C_{k+1}\setminus C_k}\Big( A+\bm^{\gamma-1}(|\bu-\bv|)\Big)\,d|\bmu|\\
&\leq \frac{c}{\nu} \int_{B_1} A\,d|\bmu|+\|\bm (|\bu-\bv|)\|^{\gamma-1} _{L^\gamma(B_1)}\|\bmu\|_{L^\gamma(B_1)}\leq  \frac{c}{\bar\nu} .
\end{flalign*}
Since the second term on the right-hand side of~\eqref{1st-est} can be estimated due to Lemma~\ref{lem:comp-sol}, we get~\eqref{eq:resc-comp-est}.

\bigskip




\noindent\textsc{Step 3. Rescaling back. } We reverse the change of variables from~\eqref{resc-comp} and analysis of the speed of the growth of $m,M,G$ leads to~\eqref{eq:comp-est}, what completes the proof.\end{proof}

\begin{coro}\label{coro:comparison}
Suppose $u\in W^{1,G}(\Omega)$ is a local SOLA to~\eqref{eq:main}, $v\in u+W^{1,G}_0(B_R)$ is a weak solution to~\eqref{eq:comp-map} on $B_R$, and parameters satisfy $q\in(0,\infty)$, $\theta\in[0,n]$, $\gamma\in(1,\infty)$.

If   $f\in L^{\gamma,\theta}(B_R)$ there exist a constant $c>0$, such that \begin{equation}
\label{eq:cominMor}\barint_{B_R} g_m (|Du-D v|)\,dx \leq c R^{-\frac{\theta}{\gamma}}\avenorm{\mu}_{L^{\gamma,\theta}(B_R)},
\end{equation}
whereas for $f\in{L^{\theta}(\gamma,q)(B_R)}$ there exist a constant $c>0$, such that \begin{equation}
\label{eq:cominMorLor}\barint_{B_R} g_m  (|Du-D v|)\,dx \leq c R^{ -\frac{\theta}{\gamma}}\avenorm{\mu}_{L^{\theta}(\gamma,q)(B_R)}.
\end{equation} 
\end{coro}
\begin{proof} Recall that $L^{\gamma,\theta}$ and $L^{\theta}(\gamma,q)$ are defined in Section~\ref{ssec:fn-sp}. Inequality~\eqref{eq:cominMor} is obtained using Proposition~\ref{prop:comp-B-xi} and then the H\"older inequality, as for~\eqref{eq:cominMorLor} we apply also Lemma~\ref{lem:LqinMarc} getting
\[\begin{split}\barint_{B_R} g_m  (|Du-D v|)\,dx &\leq c R^{ -\frac{n}{\gamma}} \avenorm{\mu}_{{\cal M}^{\gamma}(B_R)}\leq c R^{ -\frac{n}{\gamma}}\avenorm{\mu}_{L({\gamma},q)(B_R)}\leq c R^{ -\frac{\theta}{\gamma}}\avenorm{\mu}_{L^{\theta}(\gamma,q)(B_R)}.\end{split}\]
\end{proof}



\subsection{Preliminary Lorentz estimates}

We derive estimates on the maximal operator of gradient $Du$ of solutions $u=u_k$ to the problem with bounded data~\eqref{eq:main-k-for-SOLA}.  We consider the (restricted) maximal function operator related to a ball 
\begin{equation}\label{MDu-Mmu-def}
M^\ast_{2B_0}(|\mu|)(x)=\sup\limits_{\substack{x\in B_R\\ B_R\subset 2B_0}}\avenorm{ \mu }_{L^1(B_R)} 
\qquad\text{and}\qquad M^\ast_{2B_0}(f)(x)=\sup\limits_{\substack{x\in B_R\\ B_R\subset 2B_0}} \barint_{B_R} {|f|}\,dy. 
\end{equation}

The important tool for us is the following density lemma resulting from a special version Krylov \& Safonov covering lemma presented in~\cite[Lemma~13.2]{min-book}.
\begin{lem}[Krylov-Safonov density lemma]
\label{lem:covering}
Let $E,F\subset B_0\subset\rn$ be measurable sets with $B_0$ being a ball. Define\begin{equation*}
E^\kappa:=E\cap\kappa B_0\qquad\text{and} \qquad F^\kappa:=F\cap\kappa B_0
\end{equation*}
for every $\kappa\in(0,1]$. Assume that for some ${\delta}\in(0,1)$, $d\geq 1$, and $0<r_1<r_2\leq 1$ the following conditions are satisfied\begin{itemize}
\item $|E^{r_1}|\leq\frac{ {\delta}}{5^n}\left(\frac{r_2-r_1}{d}\right)^n|B_0|$
\item if $B$ is a ball such that $(d B)\subset B_0$, then $(|E\cap B|>\frac{ {\delta}}{5^n}|B|\implies B\subset F)$.
\end{itemize}
Then $ |E^{r_1}|\leq {\delta}|F^{r_2}|.$
\end{lem}

To apply the density lemma in consideration on super-level sets of the maximal operator evaluated in gradient of the solution, we need the following lemma whose proof follows the same lines as the one of~\cite[Lemma~5.2]{IC-gradest} and requires Proposition~\ref{prop:homo-problem}.
\begin{lem}\label{lem:Binc}
Suppose $u\in W^{1,G}(\Omega)$ is a weak solution to~\eqref{eq:main-k-for-SOLA}. Let   $H=H(n,G)>0$ be large fixed absolute constant. Assume further that there exist $T_0$, such that for every $T>T_0$ there exists $\ve=\ve(n,G,T)>0$, such that for every $\lambda>0$ and $B $ such that $2B\subset  B_0$ it holds that
\begin{equation*}
\Big|B\cap\{x\in B_0: \MDu(x)>HT\lambda\quad \text{ and }\quad  \Mf (x)\leq g_m(\ve\lambda)\}\Big|>\frac{|B|}{5^n G^{\chi}(HT)}.
\end{equation*}
Then
\begin{equation*}
B\subset \{x\in B_0: \MDu(x)>HT\lambda\}.
\end{equation*}
\end{lem} 

The super-level set estimates below results from Lemmata~\ref{lem:covering} and~\ref{lem:Binc} in the same way as~\cite[Proposition~5.1]{IC-gradest}. This follows the approach devised in \cite{min07, min-grad-est}. 

\begin{prop}[Super-level set estimates]
\label{prop:super-level-est} Suppose $u\in W^{1,G}(\Omega)$ is a weak solution to~\eqref{eq:main-k-for-SOLA}.  Let $B$ be a ball such that $2B\subset\subset\Omega$ and $0<r_1<r_2\leq 1$. There exist constants $H=H(n,G)>>1$ and $c(n)\geq 1 $, such that the following holds true: for every $T>1$ there exists $\ve=\ve(n,G,T)\in(0,1),$ such that
\begin{equation}
\label{eq:super-level-est}\begin{split}
&|\{x\in r_1 B: \MDu(x)>HT\lambda\}|\\ &\qquad \leq\frac{1}{G^{\chi}(HT)}|\{x\in r_2 B: \MDu(x)> \lambda\}|+|\{x\in r_1B: \Mf(x)> g_m (\ve\lambda)\}|\end{split}
\end{equation}
holds whenever\begin{equation}
\label{lambda>lambda0-def}
\lambda\geq \lambda_0:=\frac{c(n) }{(r_2-r_1)^n}\frac{G^{\chi}(HT)}{HT}\barint_{2B}|Du|\,dx
\end{equation}
\end{prop}

Let us present the Lorentz estimates.  

\begin{prop} \label{prop:maximal-est}
Suppose $u\in W^{1,G}(\Omega)$ is a weak solution to~\eqref{eq:main-k-for-SOLA} and $\chi$ is the~higher integrability exponent (see Proposition~\ref{prop:homo-problem}, {(ii)}). Then for every $(\gamma,q)\in\big[1,\chi i_G({i_m+1})/({s_Gs_m})\big)\times(0,\infty]$,
there exists a~constant $c=c(c,G,t,q)$ for which
\begin{equation}
\label{grad-est-Lor}  
\avenorm{ g_m(|Du|)}_{L(\gamma ,q  )(B/2)}\leq c\,g_m\left( \barint_{2B }|Du|dx\right) +c \avenorm{ \Mb(\mu) } _{L(\gamma ,q)(B)}
\end{equation}
holds for every $B$, such that $2B\subset\subset \Omega$.
\end{prop}

\begin{proof} We will show Lorentz estimates for the maximal operator 
\begin{equation}
\label{grad-est-1} \avenorm{ g_m ( \Mb(|Du|)) }_{L(\gamma ,q)(B/2)}\leq c\,g_m\left( \barint_{2B }|Du|dx\right) +c \avenorm{ \Mb(\mu) }_{L(\gamma ,q )(B)},
\end{equation}
which directly implies~\eqref{grad-est-Lor} via the Lebesgue differentiation theorem. First we concentrate on the case $q<\infty$ and then $q=\infty$.

\medskip

\textbf{Case $0<q<\infty$.} We apply~Proposition~\ref{prop:super-level-est}. At first we recall that $\gamma>1$, then we raise both sides of~\eqref{eq:super-level-est} to power $\frac{q}{\gamma}$, then multiply by $g_m^{q}(HT\lambda)/ \lambda$ and integrate from $\lambda_0$ given by~\eqref{lambda>lambda0-def} to $\lambda_1$. Altogether, we get 
\begin{multline*} \int_{\lambda_0}^{\lambda_1}g_m^{q }(HT\lambda) |\{x\in r_1 B: g_m(\Mb(|Du|))(x)>g_m(HT\lambda)\}|^\frac{q}{\gamma}\tfrac{d\lambda}{\lambda}\\
   \leq \frac{c}{G^{\chi q/\gamma}(HT)}\int_{\lambda_0}^{\lambda_1}g_m^{q}(HT\lambda) |\{x\in r_2 B:  \Mb(|Du|)(x)> \lambda\}|^\frac{q}{\gamma}\tfrac{d\lambda}{\lambda}\\
   + c\int_{\lambda_0}^{\lambda_1}g_m^{q}(HT\lambda)  |\{x\in r_1B: \Mb(\mu)(x) >  g_m(\ve\lambda) \}|^\frac{q}{\gamma}\tfrac{d\lambda}{\lambda}.
\end{multline*}
Therefore, changing variables and Lemma~\ref{g<lambda} imply
\begin{multline*}  \int_{g_m(HT\lambda_0)}^{g_m(HT\lambda_1)}\lambda^{q} |\{x\in r_1 B: g_m(\Mb(|Du|)(x))> \lambda\}|^\frac{q}{\gamma}\tfrac{d\lambda}{\lambda}\\
   \leq  {c }\left(\frac{ (HT)^{\frac{s_Gs_m}{i_m+1}} }{G^{\chi/\gamma}(HT)}\right)^{q}\int_{\lambda_0}^{\lambda_1}g_m^q(\lambda)|\{x\in r_2 B: g_m( \Mb(|Du|)(x))> g_m(\lambda)\}|^\frac{q}{\gamma}\tfrac{d\lambda}{\lambda}\\
   + c \left(\frac{ HT }{\ve}\right)^{q{\frac{s_Gs_m}{i_m+1} }}\int_{\lambda_0}^{\lambda_1}g_m^q(\ve\lambda) |\{x\in r_1B: \Mb(\mu)(x) >  g_m(\ve\lambda) \}|^\frac{q}{\gamma}\tfrac{d\lambda}{\lambda}\\
   \leq  {c }\left(\frac{ (HT)^{\frac{s_Gs_m}{i_m+1}} }{G^{\chi/\gamma}(HT)}\right)^{q}\int_{g_m(\lambda_0)}^{g_m(\lambda_1)} \lambda^q|\{x\in r_2 B:  g_m( \Mb(|Du|))(x)>  \lambda \}|^\frac{q}{\gamma}\tfrac{d\lambda}{\lambda}\\
   + c \left(\frac{ HT }{\ve}\right)^{q{\frac{s_Gs_m}{i_m+1} }}\int_{g_m(\lambda_0)}^{g_m(\lambda_1)}\lambda^{\gamma  }|\{x\in r_1B: \Mb(\mu)(x) >  \lambda \}|^\frac{q}{\gamma}\tfrac{d\lambda}{\lambda}.
\end{multline*}

We add to both sides the quantity  
\begin{multline*} \int_0^{g_m(HT\lambda_0)} \lambda^{q}|\{x\in r_1 B: g_m(\Mb(|Du|))(x)> \lambda\}|^\frac{q}{\gamma}\tfrac{d\lambda}{\lambda} \leq \tfrac{1}{q +1} g_m^q(HT\lambda_0) |B|^\frac{q}{\gamma}\\
 \leq c(T) |B|^\frac{q}{\gamma}\left(\tfrac{1}{ r_2-r_1  } \right)^{n q \frac{s_Gs_m}{i_m+1}}g_m^q\left(\barint_{2B }|Du|\,dx\right) ,
\end{multline*}
estimated in the above way due to definition of $\lambda_0$~\eqref{lambda>lambda0-def}. Let us prepare for application of Lemma~\ref{lem:absorb1} with $R=1$ and
\[\phi(\kappa)= \int_0^{g_m(HT\lambda_1)}\lambda^{q}|\{x\in \kappa B: g_m(\Mb(|Du|))(x)> \lambda\}|^\frac{q}{\gamma}\tfrac{d\lambda}{\lambda}\qquad\text{for}\qquad \kappa\in(0,1].\] For $H=H(n,G)>0$ is a fixed (big) constant we  choose $T_0=T_0(n,G)$ big enough for $T>T_0$ to satisfy \[ {c }\left(\frac{ (HT)^{\frac{s_Gs_m}{i_m+1}} }{G^{\chi/\gamma}(HT)}\right)^{q}\leq\frac{1}{2}\qquad\text{which is possible since }\quad \gamma<\chi\frac{i_G}{s_G}\frac{i_m+1}{s_m}.\]
Moreover, we use that
\begin{equation*}\begin{split}
\int_0^{\infty}\lambda^{q}|\{x\in  B:  \Mb(|\mu|)(x) > \lambda\}|^\frac{q}{\gamma}\tfrac{d\lambda}{\lambda} 
&=\tfrac{1}{q}\| \Mb(\mu) \|^{q}_{L(\gamma,q)(B)}.\end{split}
\end{equation*}
To do it we sum up the above remarks we have
\[\phi(r_1)\leq \tfrac{1}{2}\phi(r_2)+ c(T) |B|^\frac{q}{t}\left(\tfrac{1}{ r_2-r_1  } \right)^{n q \frac{s_Gs_m}{i_m+1}}g_m^q\left(\barint_{2B }|Du|dx\right) +c\|\Mb (\mu)\|^{q}_{L(t,q)(B)} \]
and Lemma~\ref{lem:absorb1} implies
\[\int_0^{g_m(HT\lambda_1)}\lambda^{q}|\{x\in  B/2: g_m(\Mb(|Du|)(x))> \lambda\}|^\frac{q }{\gamma }\tfrac{d\lambda}{\lambda}\leq  c  |B|^\frac{q}{\gamma} g_m^q\left(\barint_{2B }|Du|dx\right) +c \|\Mb(\mu)\|^{q}_{L(\gamma,q)(B)} .\]
Now we let $\lambda_1\to\infty$ and get~\eqref{grad-est-1} for $0<\gamma<\infty$.

\bigskip

\textbf{Case $q=\infty$.} We we also apply Proposition~\ref{prop:super-level-est} with $\gamma\geq 1$. Multiplying both sides of~\eqref{eq:super-level-est} by $g_m^\gamma(HT\lambda)$ and computing supremum  we get
\begin{multline*}
\sup_{\lambda_0\leq\lambda\leq \lambda_1} g_m^\gamma(HT\lambda) |\{x\in r_1 B: g_m(\Mb(|Du|)(x))>g_m(HT\lambda)\}|\\
 \leq\frac{ (HT)^{\gamma\frac{s_Gs_m}{i_m+1}} }{G^{\chi}(HT)}\sup_{\lambda_0\leq\lambda\leq \lambda_1} g_m^\gamma(\lambda)|\{x\in r_2 B: g_m(\Mb(|Du|)(x))> g_m(\lambda)\}|\\
  +\left(\frac{HT}{\ve}\right)^{\gamma\frac{s_Gs_m}{i_m+1}} \sup_{\lambda_0\leq\lambda\leq \lambda_1} g_m^\gamma(\ve\lambda)|\{x\in r_1B: \Mb(\mu)(x) > g_m(\ve\lambda)\}|.\end{multline*}
As in the case of finite $q$ we change the variables, use the definition of $\lambda_0$, and add  to both sides the initial term
\[\sup_{0\leq\lambda\leq g_m(HT\lambda_0)}\lambda^{\gamma}|\{x\in r_2 B: g_m(\Mb(|Du|)(x))>\lambda\}|\leq g_m^\gamma(HT\lambda_0) |B|\leq \frac{c|B| }{ \left(r_2-r_1  \right)^{n \gamma }}\, g_m^\gamma\left(\barint_{2B }|Du|dx\right) \]
to obtain
\begin{equation}
\begin{split}\label{to-est-inf}
&\sup_{ 0\leq \lambda\leq g_m(HT\lambda_1)}\lambda^{\gamma}|\{x\in r_1 B: g_m(\Mb(|Du|)(x))> \lambda\}| 
 \leq  \frac{c|B| }{ \left(r_2-r_1  \right)^{n \gamma }}\  g_m^\gamma\left(\barint_{2B }|Du|dx\right) \\
&\quad\qquad\qquad\qquad\qquad+\frac{ (HT)^{\gamma\frac{s_Gs_m}{i_m+1}} }{G^{\chi}(HT)}\sup_{ 0\leq\lambda\leq g_m(HT\lambda_1)}\lambda^{\gamma  }|\{x\in r_2 B: g(\Mb(|Du|)(x))> \lambda\}|\\
&\quad\qquad\qquad\qquad\qquad+\left(\frac{HT}{\ve}\right)^{\gamma }\sup_{ 0\leq\lambda\leq \lambda_1}\lambda^{\gamma }|\{x\in r_1B: \Mb (\mu)(x) > \lambda\}|.\end{split}
\end{equation}
We apply Lemma~\ref{lem:absorb1} with $R=1$ and
\[\phi(\kappa)= \sup_{0\leq \lambda\leq HT\lambda_1}\lambda^{\gamma }|\{x\in \kappa B: g_m(\Mb(|Du|)(x))> \lambda\}| \qquad\text{for}\qquad \kappa\in(0,1].\]
Note that due to the upper bound on $\gamma$, we can choose $T_0$, such that ${ (HT)^{\gamma\frac{s_Gs_m}{i_m+1}} }/{G^{\chi}(HT)}\leq\frac{1}{2}.$ Since $H$ is a  constant and $\sup_{\lambda>0}\lambda^\gamma |\{x\in r_1 B: \Mb(|\mu|)(x)> \lambda\}|=\|\Mb(|\mu|)\|^\gamma_{{\cal M}^\gamma(r_1 B)}$,
from~\eqref{to-est-inf} we get 
\[\phi(r_1)\leq\frac{1}{2}\phi(r_2)+\frac{c|B| }{ \left(r_2-r_1  \right)^{n \gamma }} g_m^\gamma\left(\barint_{2B }|Du|dx\right) +\|\Mb(\mu)\|^\gamma_{{\cal M}^\gamma(r_1 B)}.\]
Therefore, Lemma~\ref{lem:absorb1} implies that
\[\sup_{0\leq \lambda\leq HT\lambda_1}\frac{\lambda^{\gamma }}{|B/2|}|\{x\in   B/2: g_m(\Mb(|Du|)(x))> \lambda\}|\leq c g_m^\gamma\left(\barint_{2B }|Du|dx\right)+c\| \Mb(\mu)\|^\gamma_{{\cal M}^\gamma(B)}.\]
To conclude~\eqref{grad-est-1} in the case of $q=\infty$ it suffices to let $\lambda_1\to\infty$.
\end{proof}

\subsection{Preliminary Morrey estimates}

\begin{prop}[Preliminary Morrey estimates]\label{prop:grad-u-est}
Suppose $u\in W^{1,G}(\Omega)$ is a weak solution to~\eqref{eq:main-k-for-SOLA} and $\gamma$ and $\theta$ satisfy~\eqref{q-Morrey}, then there exist a constant $c=c(n,\nu,s_G)>0$, such that
\begin{equation}\label{eq:Morrey-1st-est}
 [ g_m(|Du|)]_{L^{1,\frac{\theta}{\gamma}}(\Omega_1)}\leq c\big({\rm dist}(\Omega_1,\pa \Omega_2)\big)^{ \frac{\theta}{\gamma }-n} \|g_m(|Du|)\|_{L^1(\Omega_2)}+c\|\mu\|_{L^{\gamma,\theta}(\Omega_2)}. \end{equation}
\end{prop}

In fact we will show the result in the broader range of parametrs than~\eqref{q-Morrey}. Namely, the above estimate holds provided
\begin{equation}
\label{wt-q-Morrey}i_G\leq \theta\leq n\qquad\text{and}\qquad 1<\gamma<\frac{\theta i_G\wt{\chi}}{\theta s_G-\theta+i_G\wt{\chi}} 
\end{equation} 
with some $\wt{\chi}=\wt{\chi}(n,i_G,s_G,\nu,L)>1$.

\begin{proof} 
When  $v\in u+W^{1,G}_0(B_R)$ is a solution to~\eqref{eq:comp-map} on $B_R$, we have 
\begin{equation*}
\begin{split}
\int_{B_\vr}g_m(|Du|)\,dx&\leq 
c\int_{B_\vr}g_m(|Du-Dv|)\,dx+
c\int_{B_\vr}g_m(|Dv|)\,dx\\
&\leq 
c\int_{B_R}g_m(|Du-Dv|)\,dx+c
\left(\frac{\vr}{R}\right)^{n-\beta_m}\int_{B_R}g_m(|Dv|)\,dx\\
&\leq 
c\int_{B_R}g_m(|Du-Dv|)\,dx+
c\left(\frac{\vr}{R}\right)^{n-\beta_m}\int_{B_R}g_m(|Du|)\,dx.
\end{split}\end{equation*}
We use above the Jensen inequality, extend the domain of the integration, apply Proposition~\ref{prop:homo-problem} {\it (ii)}, and the fact that $v$ is the solution to the homogeneous problem and thus a minimiser to the~variational formulation.

Let us denote \begin{equation}
\label{gamma:delta}
\wt{\chi}=\min\{\chi,1/\beta_m\},\qquad\sigma=n-\tfrac{\theta}{\gamma},\qquad\text{and}\qquad \delta=n-\beta_m\gamma\tfrac{s_ms_G}{i_m+1},
\end{equation} where $\chi$ is the higher integrability exponent coming from Proposition~\ref{prop:homo-problem}  {\it (ii)} and $\beta_m$ comes from Proposition~\ref{prop:homo-problem}  {\it (iv)}. Recall that $s_G-1<s_m$, $s_G\gamma<\theta$ and notice that
\[\gamma\leq\frac{\theta \wt{\chi}}{\theta s_G-\theta+i_G\wt{\chi}}\leq \frac{\theta }{\beta_m\gamma\frac{s_ms_G}{i_m +1}+1}.\] Then $\sigma<\delta$ and we can apply Lemma~\ref{lem:absorb2} with \[\phi(\vr)=
\int_{B_\vr}g_m(|Du|)\,dx,\qquad {\cal B}=\frac{1}{R^{\sigma}}\int_{B_R}g_m(|Du-Dv|)\,dx,\]
 to get
\[\int_{B_\vr}g_m(|Du|)\,dx\leq c\left(\frac{ \vr}{R}\right)^{\sigma}\left\{ \int_{B_R}g_m(|Du|)\,dx+  
\int_{B_R}g_m(|Du-Dv|)\,dx\right\}\quad\text{for }\ \vr\leq R\leq \bar{R}.\]
Therefore, by recalling $\sigma$ from~\eqref{gamma:delta} we end with
\[\vr^{\frac{\theta}{\gamma }-n}\int_{B_\vr}g_m(|Du|)\,dx\leq cR^{ \frac{\theta}{\gamma }-n}\left\{\|g_m(|Du|)\|_{L^1(B_R)}+  
\int_{B_R}g_m(|Du-Dv|)\,dx\right\}\]
and consequently, for every $\Omega_1\subset\subset\Omega_2 \subset\subset\Omega$, taking into account~\eqref{eq:cominMor}, we get~\eqref{eq:Morrey-1st-est}.
\end{proof}

\section{Proofs of main results}\label{sec:proofs}
This section is devoted to presentation of main proofs and then providing a compact discussion. 
\subsection{Proofs}\label{ssec:main-proofs}
\begin{proof}[Proof of Theorem~\ref{theo:Lorentz-est}] For approximate solutions $u_k$ to~\eqref{eq:main-k-for-SOLA}, Proposition~\ref{prop:maximal-est} provides the following inequality 
\[\begin{split} \avenorm{ g_m( |Du_k|)}_{L\left(\gamma,q\right)(B_{R/4})}&\leq  c\,g_m\Big(\barint_{B_R}|Du_k|dx\Big)+c\avenorm{\mu }_{L(\gamma,q)(B_R)}\quad\text{for all }\ {B_R\subset\subset\Omega}.\end{split}\] We can pass to the limit with $k\to\infty$ according to assumptions on SOLA (Definition~\ref{def:SOLA}). The  proof of~\eqref{eq:Lorentz-est} can be concluded by a standard covering argument.
\end{proof}

\begin{proof}[Proof of Theorem~\ref{theo:Morrey-est}] As in the previous proofs we shall consider $u_k$ solving~\eqref{eq:main-k-for-SOLA} and after getting the estimates pass to the limit with $k\to\infty$ according to assumptions on SOLA (Definition~\ref{def:SOLA}). We skip $k$ in notation.

 We provide first the estimates for a problem defined on a unit ball $B_1\subset\rn$ and then rescale it to obtain the final estimates. Let us consider $\hat{u}$ solving\begin{equation}
\label{eq:tilde}
 -\dv\,\hat{a}(x,D\hat{u})=\hat{\mu}\in L^\gamma( B_1).
\end{equation}
Proposition~\ref{prop:maximal-est} with $q=\gamma$ yields then
\[\avenorm{ g_m( |D\hat{u}| )}_{L^{\gamma}(B_{1/8})}\leq c\,g_m\Big(\barint_{B_{1/2}}|D\hat{u}|dx\Big)+c\avenorm{ M^*_{B_{1/2}}(\hat{\mu})}_{L^{\gamma}(B_{1/2})}.
\]
We estimate the right-hand side above using the Jensen inequality and properties of Morrey norm to get
\begin{equation}
\label{est:tilde}\begin{split}\avenorm{ g_m( |D\hat{u}| )}_{L^{\gamma}(B_{1/8})}&\leq c[g_m\left(|D\hat{u}| \right)]_{L^{1,\frac{\theta-\gamma}{\gamma}}(B_{1})}+c\avenorm{\hat{\mu}}_{L^{\gamma,\theta}(B_{1/2})}.\end{split}
\end{equation}  

Going back to the original solution $u$ we consider a ball $B_\vr=B(x_0,\vr)\subset\subset \Omega$ and rescale the problem. For $y\in B_1$ we put \[\hat{u}(y) :=u(x_0 +\vr y)/\vr,\quad \hat{\mu}(y) := \vr \mu(x_0 +\vr y),\quad\text{and} \quad\hat{a}(y,z)=a( x_0 + \vr y, z).\] Notice that $\hat{u}$ solves~\eqref{eq:tilde} and we have the estimate~\eqref{est:tilde} for it. Using Remark~\ref{rem:Mor-scal} we infer the estimate for $u$
\begin{equation*}
  \avenorm{ g_m( |D {u}|)}_{L^{\gamma}(B_{\vr/8})} \leq c\Big\{[g_m\left(|D {u}| \right)]_{L^{1,\frac{\theta-\gamma}{\gamma}}(B_{\vr})}+\| {\mu}\|_{L^{\gamma,\theta}(B_{\vr})}\Big\}\vr^\frac{\theta}{\gamma},
\end{equation*} 
which by standard covering argument and then by Proposition~\ref{prop:grad-u-est} implies
\begin{equation*}
\| g_m( |D {u}|)\|_{L^{\gamma,\theta}(\Omega_1)}  \leq c \|g_m(|D {u}|)\|_{L^{1}({\Omega_2})}+\|\mu\|_{L^{\gamma,\theta}({\Omega_2})} 
\end{equation*} 
with for $\Omega_1\subset\subset\Omega_2\subset\subset\Omega$ and $c=c(n,m,G,\theta, {\rm dist}(\Omega_1 , \pa\Omega_2 ))$. 

In order to conclude we again use the same scaling argument. For $y\in B_1$ we put 
\[ \bar{\bar{u}}(y) :=u(x_0 +Ry )/R,\quad  \bar{\bar{\mu}}(y) := R \mu(x_0 +R y),\quad\text{and}\quad  \bar{\bar{a}}(y,z)=a( x_0 + R y, z).\]
and we have $\| g_m( |D  \bar{\bar{u}}|)\|_{L^{\gamma,\theta}(B_{3/4})}  \leq c \|g_m(|D  \bar{\bar{u}}|)\|_{L^{1}({B_1})}+\|  \bar{\bar{\mu}}\|_{L^{\gamma,\theta}({B_1})}.$ Consequently, again by Remark~\ref{rem:Mor-scal}, we have 
\[\| g_m( |D  {u}|)\|_{L^{\gamma,\theta}(B_{3R/4})}  \leq c R^{\frac{\theta}{\gamma}-n}\|g_m(|D  {u}|)\|_{L^{1}(B_{R})}+\|\mu\|_{L^{\gamma,\theta}({B_R})}\]
and the desired estimate follows.
\end{proof}

\begin{proof}[Proof of Theorem~\ref{theo:Bord-Morrey-est}] The proof follows precisely the same lines as above but using before~\eqref{est:tilde} that $\|M_B^*(\mu)\|_{L^1(B)}\leq \|\mu\|_{L\log L(B)}$. 
\end{proof}

\begin{proof}[Proof of Theorem~\ref{theo:Lor-Mor}] Starting as in the proof of Theorem~\ref{theo:Lorentz-est} with Corollary~\ref{coro:comparison} we get
\[\begin{split} \avenorm{ g_m( |Du_k|)}_{L\left(\gamma,q\right)(B_{R/4})}&\leq c\,g_m\left(\barint_{B_R}|Du_k|dx\right)+c\avenorm{\mu}_{L^\theta(\gamma,q)(B)}\quad\text{for all }\ {B_R\subset\subset\Omega}.\end{split}\] Then we use it instead of~\eqref{est:tilde} in the reasoning of the proof of Theorem~\ref{theo:Morrey-est} to get the claim.
\end{proof}
\subsection{Remarks}
\begin{rem}\rm
An inspection of the proofs of Theorems~\ref{theo:Lorentz-est} and~\ref{theo:Morrey-est}, and~\ref{theo:Lor-Mor} indicates that we actually prove  the expected result in the slightly wider range of parameters. For brevity the main claims are formulated under closed-ended condition capturing the interesting end points. In fact, the proof of Lorentz estimates from Theorem~\ref{theo:Lorentz-est} works provided
\begin{equation*}
 1<\gamma<\frac{n i_G {\chi}}{n s_G-n+i_G {\chi}} \qquad\text{and}\qquad
0<s\leq\infty
\end{equation*} 
with $ {\chi}$ being the higher integrability exponent $ {\chi}= {\chi}(n,i_G,s_G,\nu,L)>1$, whereas the proof of Lorentz estimates from Theorem~\ref{theo:Lorentz-est} under the corresponding restrictions~\eqref{wt-q-Morrey} broader.
\end{rem}
\begin{rem}\rm
\label{rem:gn} Gagliardo--Nirenberg--type inequality in the Orlicz setting from \cite[Theorem~4.3]{kapipa} specialized to the case of derivatives of order $0, 1$ and $2$ and involving integrability of $g_m^\gamma(|Du|),$ $\gamma\geq 1$ provides inequality
\begin{equation}
\label{gn-inq}
\|Du\|_{L_{g^\gamma_m}}\leq c\sqrt{\|u \| _{L_{\Phi_1}}}\sqrt{\|D^{(2)}u\|_{L_{\Phi_2}}} \approx c\sqrt{\|u \| _{L_{m^\gamma}}}\sqrt{\|D^{(2)}u\|_{L^\gamma}} 
\end{equation}
with $c>0$ independent of $u$. To be precise, using the  notation therein, for an $N$-function $F(\lambda)=M^\frac{1}{2}(\lambda^2),$ $\Phi_1(\lambda)=g_m^\gamma(F(\sqrt{\lambda}))$ and $ \Phi_2(\lambda)=g_m^\gamma(F^\ast(\sqrt{\lambda}))$. Thus, to infer the above equivalence one should notice that having $M$ doubling it holds that $F^\ast(\lambda) \approx (M^\ast)^\frac{1}{2}(\lambda^2).$ This is indeed the consequence of the definition of the complementary function, doubling condition, and the elementary observation that $\sup_{t>0}\sqrt{h(t)}\leq \sqrt{\sup_{t>0}h(t)}$ applied to $F(\lambda^2)\geq 0$ and to $M(\sqrt{\lambda})\geq 0$.\\
Let us mention that \cite[Theorem~4.3]{kapipa}  provides a weighted version of~\eqref{gn-inq}.
\end{rem}

\appendix
\addcontentsline{toc}{section}{Appendices}

\section{Appendix}

\subsection{Function  spaces}\label{ssec:fn-sp}
In this section we define and present basic properties of several function  spaces, which are taken into account in the paper. In every definition $\Omega\subset\rn$ is assumed to be an open subset. By the local versions of the spaces defined in this section, we mean naturally those where the norm is finite on arbitrary compact subset of $\Omega$.

\begin{defi}[Lorentz and Marcinkiewicz space]\label{def:Lor:sp}
Let $q,\gamma>0$. A measurable map $f : \Omega\to\r^k$, $k\in\n$ belongs to the Lorentz space $L(q,\gamma)(\Omega)$ if and only if
\[\|f\|_{L(q,\gamma)(\Omega)}= \left(q \int_0^\infty
 \lambda^\gamma |\{x\in\Omega: |f (x)|>\lambda\}| ^\frac{\gamma}{q}
\tfrac{d\lambda}{\lambda}\right)^\frac{1}{\gamma}<\infty.\]
The Marcinkiewicz space ${\cal M}^q(\Omega)=
L(q,\infty)(\Omega)$ is defined setting 
\[\|f\|_{{\cal M}^q(\Omega)}=\sup_{\lambda>0}\lambda|\{x\in\Omega: |f (x)|>\lambda\}|^\frac{1}{q}.\]
\end{defi}
Let us point out that the Lorentz spaces are intermediate to the Lebesgue spaces in the following sense: for $0<q<t<r\leq\infty$ 
\[L^r=L(r,r)\subset L(t,q)\subset L(t,t)=L^t\subset L(t,r)\subset L(q,q)= L^q.\]
In particular, $L^p\subset {\cal M}^p\subset L^{p-\ve}$, where the inclusions are proper (for the second one consider a function $|x|^{-n/p}$).

\medskip

We shall make use of the following averaged norms
\[\begin{split}\avenorm{f}_{L(q,\gamma)(\Omega)} = \left(q \int_0^\infty
\lambda^\gamma\left(\tfrac{ |\{x\in\Omega: |f (x)|>\lambda\}|}{|\Omega|}\right)^{\frac{\gamma}{q}}
\tfrac{d\lambda}{\lambda}\right)^\frac{1}{\gamma} \quad\text{and}\quad
\avenorm{f}_{{\cal M}^q(\Omega)} = \sup_{\lambda>0}\tfrac{\lambda }{|\Omega|}|\{x\in\Omega: |f (x)|>\lambda\}| ^\frac{1}{q}.\end{split}\]

\begin{defi}[Morrey space]\label{def:Mor:sp} Let $q\geq 1$ and $\theta\in[0,n]$.  A measurable map $f : \Omega\to\r^k$, $k\in\n$ belongs to the Morrey space $L^{q,\theta}(\Omega)$ 
if and only if \[\|f\|_{L^{q,\theta}(\Omega)}:= \sup\limits_{\substack{
B(x_0,R) \subset\rn\\ x_0\in\Omega}} R^{\frac{\theta-n}{q}} \|f\|_{L^q(B_R\cap \Omega)} <\infty.\]
\end{defi}

Combining the integrability and density conditions we consider also the followig spaces.

\begin{defi}[Lorentz-Morrey and Marcinkiewicz-Morrey spaces]\label{def:LorMor:sp} Let $q\geq 1$ and $\theta\in[0,n]$. We say that $f$ belongs to the Lorentz-Morrey space $L^{\theta}(t,q)(\Omega)$
if and only if\[\|f\|_{L^{\theta}(t,q)(\Omega)}:=\sup\limits_{\substack{
B(x_0,R) \subset\rn\\ x_0\in\Omega}}  R^\frac{\theta-n}{t}\|f\|_{L(t,q)(B_R\cap \Omega)}<\infty,\]
and, accordingly, $f$ belongs to the Marcinkiewicz-Morrey space ${\cal M}^{t,\theta}(\Omega)=L^\theta(t,\infty)(\Omega)$ if and only if\[\|f\|_{{\cal M}^{t,\theta}(\Omega)}:=\sup\limits_{\substack{
B(x_0,R) \subset\rn\\ x_0\in\Omega}}  R^\frac{\theta-n}{t}\|f\|_{{\cal M}^{t }(B_R\cap \Omega)}<\infty.\]
\end{defi}
\begin{rem}\rm \label{rem:Mor-scal} Let us consider $f\in L^{q,\theta} (B)$ with $B=B(x_0,R)$ and $\hat{f}(y) := f (x_0 + Ry)$ for $y$ from the unit ball $B_1$, it follows $[\hat{f}]_{ L^{q,\theta} (B_1)} = R^{-\theta/q} [f ]_{ L^{q,\theta} (B)}$ {and} $\|\hat{f}\|_{ L^{q,\theta} (B_1)} = R^{-\theta/q} \|f \|_{ L^{q,\theta} (B)}.$
\end{rem} 
\begin{rem}\rm We have ${\cal M}^{ q,\theta} \subset {\cal M}^q=  {\cal M}^{ q,0}$ for $q>1,\ \theta\in[0,n],$ and $L^{q,\theta}\subset{\cal M}^{q,\theta}\subset L^{t,\theta}$ for $1\leq t<q $ and $\theta\in[0,n].$  To visualise how this scale is different than the classical Lebesgue setting let us mention that despite $L^{1,0}=L^\infty$, there exist functions from $L^{1,\theta}$ for $\theta$ arbitrarily close to zero, which   do not belong to $L^q$ for any $q > 1$.
\end{rem} 
\begin{defi}[$L\log L$-spaces]\label{def:LLogL:sp} Let $\Omega\subset\rn$  be an open subset of finite measure  and $k\geq 1$. We define the space $L\log  L(\Omega)$ as a subset of integrable functions $f:\Omega\to\r^k$ such that $\int_\Omega |f | \log (e + |f |) dx <\infty,$ endowed with a norm \[\begin{split}\|f\|_{L \log  L(\Omega)}&=\inf\Big\{\lambda>0:\quad\int_\Omega\left|{f}/{\lambda}\right|\log\left(e +\left|{f}/{\lambda}\right| \right)dx\leq 1\Big\}<\infty.\end{split}\] Moreover, we define  $L \log  L^\theta(\Omega)$ as a subset of integrable functions $f:\Omega\to\r^k$ such that \[\begin{split}\|f\|_{L \log  L^\theta(\Omega)}&=\sup\limits_{\substack{ B(x_0,R) \subset\rn\\ x_0\in\Omega}}  R^{\theta-n} \|f\|_{L\log L(B_R\cap \Omega)}<\infty.\end{split}\]\end{defi}
When it is convenient we make use of localized and averaged norms
 \[\begin{split}[f]_{L^{q,\theta}(\Omega)}:= \sup_{
B_R \subset\Omega} R^{\frac{\theta-n}{q}} \|f\|_{L^q(B_R)},&\qquad \avenorm{f}_{L^{q,\theta}(\Omega)}:= \sup\limits_{\substack{B(x_0,R) \subset\rn\\ x_0\in\Omega}} R^{\frac{\theta-n}{q}} \avenorm{f}_{L^q(B_R\cap \Omega)} \\
[f]_{L^{\theta}(t,q)(\Omega)}:=\sup_{B_R \subset \Omega}  R^\frac{\theta-n}{t}\|f\|_{L(t,q)(B_R)},&\qquad 
\avenorm{f}_{L^{\theta}(t,q)(\Omega)}:=\sup\limits_{\substack{
B(x_0,R) \subset\rn\\ x_0\in\Omega}}  R^\frac{\theta-n}{t}\avenorm{f}_{L(t,q)(B_R\cap \Omega)},
\\ [f]_{{\cal M}^{t,\theta}(\Omega)}:=\sup_{B_R \subset\Omega} R^\frac{\theta-n}{t}\|f\|_{{\cal M}^{t }(B_R)},&\qquad \avenorm{f}_{{\cal M}^{t,\theta}(\Omega)}:=\sup\limits_{\substack{B(x_0,R) \subset\rn\\ x_0\in\Omega}}  R^\frac{\theta-n}{t}\avenorm{f}_{{\cal M}^{t }(B_R\cap \Omega)},\\ [f]_{{L\log L}^{\theta}(\Omega)}:=\sup_{B_R \subset\Omega} R^\frac{\theta-n}{t}\|f\|_{L\log L(B_R)},&\qquad \avenorm{f}_{{L\log L}^{\theta}(\Omega)}:=\sup\limits_{\substack{B(x_0,R) \subset\rn\\ x_0\in\Omega}}  R^\frac{\theta-n}{t}\avenorm{f}_{{L\log L}(B_R\cap \Omega)}.\end{split}\]


 \subsection{Basics}

The classical reference for this section is~\cite{adams-hedberg}, most of the needed estimates can be found in~\cite{min-grad-est}. The same framework with slightly more comments is presented in~\cite{IC-gradest}.
 
 Let us present basic embedding.

\begin{lem}[Lemma~6,~\cite{min-grad-est}]\label{lem:LqinMarc}
Suppose $A\subset\rn$ is measurable. If $h\in{\cal M}^t(A)$, then $h\in L^q(A)$ for every $q\in[1,t)$ and
\[\|h\|_{L^q(A)}\leq\left(\tfrac{t}{t-q}\right)^\frac{1}{q}|A|^{\frac{1}{q}-\frac{1}{t}}\|h\|_{{\cal M}^t(A)}. \]
\end{lem}

We shall use two different classical absorption lemmas, both to be find in~\cite{Giusti}.
\begin{lem}[\cite{Giusti}, Lemma~6.1]
\label{lem:absorb1}
Let $\phi:[R/2,3R/4]\to[0,\infty)$ be a function such that
\[\phi(r_1)\leq \tfrac{1}{2}\phi(r_2)+\mathcal{A}+\tfrac{\mathcal{B}}{(r_2-r_1)^\beta}\qquad\text{for every}\qquad \tfrac R2\leq r_1<r_2\leq \tfrac{3{R}}{4}\]
with $\mathcal{A,B}\geq 0$ and $\beta>0$. Then there exists $c=c(\beta)$, such that $\ \phi(R/2)\leq c \left(\mathcal{A}+{\mathcal{B}}{R^{-\beta}}\right).$
\end{lem}
\begin{lem}[\cite{Giusti}, Lemma~7.3]
\label{lem:absorb2}
Let $\phi:[0,\bar{R}]\to[0,\infty)$ be a non-decreasing function such that
\[\phi(\vr)\leq c_0\left(\tfrac{\vr}{R}\right)^{\delta }\phi(R)+ {\mathcal{B}}R^{\sigma}\quad\text{for every }\vr\leq R\leq \bar{R},\]
with some $0<\sigma<\delta $ and ${\mathcal{B}}>0$. Then there exists $c=c(c_0,\sigma)$, such that\[\phi(\vr)\leq c\left\{\left(\tfrac{\vr}{R}\right)^{\sigma}\phi(R)+  {\mathcal{B}}\vr^{\sigma}\right\}\quad\text{for every }\vr\leq R\leq \bar{R}.\]
\end{lem}



\begin{lem}\label{g<lambda} Suppose $H$ is an increasing and convex function $H\in C^1(0,\infty)$ satisfying $\Delta_2$-condition.\\ If $h(t)=H'(t)$, then there exists a constant $c$, such that for every $t>0$ and $\lambda>1$, we have
\[h(\lambda t)\leq s_H \lambda^{s_H-1} h(t).\]
\end{lem}


\end{document}